\documentclass[12pt,oneside]{amsart}
\usepackage{epsf}
\usepackage{amsmath,amscd}
\usepackage{amsmath, latexsym, amsfonts, amssymb, amsopn, amscd, mathtools, esvect, bigints, graphicx, float, indentfirst, inputenc, amssymb, pdfpages, multirow, multicol, physics, url,alltt, mathrsfs, amsthm, tikz-cd,tikz, pgfplots}
\usepackage{import, xifthen, transparent}
\usepackage{etoolbox, adjustbox, natbib}
\usepackage{enumitem, stmaryrd, subfiles, young, appendix}
\usetikzlibrary{calc}
\usepackage{subcaption}
\usepackage{hyperref,empheq}

\numberwithin{figure}{section}

\newsavebox{\savepar}

\theoremstyle{plain}
\newtheorem{thm}{Theorem}[section]
\newtheorem*{thm*}{Theorem}
\newtheorem{Def}[thm]{Definition}
\newtheorem{prop}[thm]{Proposition}
\newtheorem{Ex}[thm]{Example}
\newtheorem{cor}[thm]{Corollary}
\newtheorem{lem}[thm]{Lemma}
\newtheorem{cl}[thm]{Claim}

\newtheorem{rem}[thm]{Remark}

\newtheorem{prob*}{Problem}

\newcommand{\EE}{\mathcal{E}}

\newcommand{\FF}{\mathcal{F}}

\newcommand{\II}{\mathcal{I}}

\newcommand{\RRR}{\mathbb{R}}

\newcommand{\CCC}{\mathbb{C}}
\newcommand{\ZZZ}{\mathbb{Z}}

\renewcommand{\tilde}{\widetilde}

\renewcommand{\bar}{\overline}

\makeatletter
\renewcommand{\xleftrightarrow}[2][]{\ext@arrow 9999{\longleftrightarrowfill@}{#1}{#2}}
\newcommand{\longleftrightarrowfill@}{\arrowfill@\leftarrow\relbar\rightarrow}
\makeatother

\newcommand*\widefbox[1]{\fbox{\hspace{2em}#1\hspace{2em}}}

\newcommand{\im}{\mathrm{im}\,}

\newcommand{\Gr}{\operatorname{Gr}}

\newcommand{\diag}{\operatorname{diag}}

\makeatletter
\renewcommand\subsection{\@startsection{subsection}{3}{\z@}%
	{-3.25ex\@plus -1ex \@minus -.2ex}%
	{1.5ex \@plus .2ex}
	{\normalfont\normalsize\bfseries}}
\makeatother

\makeatletter
\renewcommand\subsubsection{\@startsection{subsubsection}{3}{\z@}%
	{-3.25ex\@plus -1ex \@minus -.2ex}%
	{1.5ex \@plus .2ex}
	{\normalfont\normalsize\bfseries}}
\makeatother

\pgfplotsset{compat=1.18}

\begin{document}
\title{Principal Minors and Absolute Values}

\author[F. Villacis]{Francisco Villacis}
\address{Department of Pure Mathematics\\ University of Waterloo\\
Waterloo, Ontario \; N2L 3G1 \\Canada}
\email{fvillacis@uwaterloo.ca}

\thanks{The author was partially supported by a Queen Elizabeth II Graduate Scholarship in Science and Technology (QEII-GSST)}

\begin{abstract}
	We prove that the principal minors of an orthogonal projection matrix determine the coordinatewise absolute value of its image. More precisely, if $W_1,W_2\subset \mathbb{C}^n$ are subspaces whose orthogonal projection matrices have equal corresponding principal minors of all orders, then $|W_1|=|W_2|$, where $|W|$ is the image of $W$ under the coordinatewise absolute value map. The proof is divided into different cases related to the connectivity of the associated matroids. We provide a stratification of the Grassmannian in terms of the connectivity levels and give a structure theorem for the elements of each stratum.
\end{abstract}

\maketitle


\section{Introduction}\label{sec:Intro}
	Given an $n\times n$ orthogonal projection matrix $A$, we uniquely obtain a subspace $W\coloneqq \im A\subset \CCC^n$ and we can consider the set
    \[
        |W|=\left\{(|w_1|,\dots ,|w_n|)\in \RRR^n_{\ge 0}: (w_1,\dots, w_n)\in W\right\},
    \]
    where $|w_i|$ is the absolute value of the $w_i\in \CCC$.
    In this paper we provide a proof of the following theorem.
    
	\begin{thm}
    	If $A$ and $B$ are two $n\times n$ orthogonal projection matrices that have the same corresponding principal minors of all orders, then
        \[
            |\im A|=|\im B|.
        \]
        \label{thm:same-principal-minors-have-equal-image}
    \end{thm}
    \vspace{-0.6cm}
    \noindent Starting with a subspace $W\subset \CCC^n$, we obtain a preferred orthogonal projection matrix corresponding to $W$, namely, the matrix representation of the orthogonal projection operator $\CCC^n\twoheadrightarrow W\subset \CCC^n$ with respect to the standard basis of $\CCC^n$ endowed with the inner product
    \[
    	\langle u,v\rangle=\sum_{i=1}^n\bar{u}_iv_i.
    \]
    Throughout, we identify all operators on $\CCC^n$ with its corresponding $n\times n$ matrix  representation written with respect to the standard basis and orthogonality will always be with respect to the inner product above.
    
    Theorem \ref{thm:same-principal-minors-have-equal-image} is closely related to a theorem of Loewy \cite{Loewy1986} regarding the classification of matrices that have equal principal minors of all orders. To see this, let $A$ and $B$ denote the orthogonal projection matrices onto subspaces $W_1,W_2\subset \CCC^n$ respectively. Then Loewy's theorem says that, under certain rank assumptions, if $A$ and $B$ have the same principal minors, then $B$ is a conjugate of $A$ or its transpose by a diagonal matrix. We show in Section \ref{sec:Results-from-Lin-Alg} that this matrix can be taken to be unitary and hence, $|W_1|=|\im A|=|\im B|=|W_2|$. 
    
    Loewy's theorem assumes two rank conditions: first, that $A$ is irreducible, and second, that it has no ``cuts''. In Section \ref{sec:background} we discuss this theorem in more detail, its relation to Theorem \ref{thm:same-principal-minors-have-equal-image} and describe some related problems.
    
    The proof of Theorem \ref{thm:same-principal-minors-have-equal-image} is done by induction on $n$. We prove the base cases $n=1,2,3$ first, and then for $n\ge 4$ we consider the three cases mentioned above. The reducible case follows directly from the induction hypothesis, while the no-cut case follows from Loewy's theorem. The focus of this paper is the case of irreducible orthogonal projection matrices with a cut. In Section \ref{sec:Results-from-Lin-Alg} we set up the required machinery for the proof.
    More precisely, we obtain a structure theorem for the images of irreducible orthogonal projection matrices that have a cut and obtain an explicit description of the principal minors of these. Irreducibility and cuts of orthogonal projection matrices are related to the notion of $k$-connectivity from matroid theory.
    The decomposition we obtain corresponds to writing an exactly 2-connected matroid as a 2-sum of smaller matroids, and this lets us describe $|W|$ in terms of the smaller dimensional subspaces in the decomposition.
    In Section \ref{sec:proof-of-thm-11}, we provide the full proof of Theorem \ref{thm:same-principal-minors-have-equal-image}.

     In Section \ref{sec:stratification} we make explicit the connection between irreducibility and having no cuts with that of $k$-connectivity. We show that a cut corresponds to the notion of a $2$-separation from matroid theory.
    We use this to extend the hypotheses of Loewy to a more general definition of exactly $k$-connected subspaces and $\infty$-connected subspaces.
    This gives a stratification of the Grassmannian, which is coarser than the one of Gelfand-Goresky-MacPherson-Serganova \cite{GelfandGoreskyMacPhersonSerganova}. We conclude by explicitly computing the open stratum using well-known results from matroid theory. Appendix \ref{sec:structure-of-k-conn-subspaces} has a self-contained treatment of these results from the linear algebra point of view and we prove a relation between the connectivity of subspaces with the cosine-sine decomposition of certain matrices.
    
    \subsection{Notation}
		\begin{enumerate}
			\item We write $[n]\coloneqq \{1,\dots, n\}$ and only work over $\CCC$. All our linear operators on $\CCC^n$ are identified with their corresponding matrix representation written with respect to the standard basis.
			\item All our vectors in $\CCC^n$ are column vectors written with respect to the standard basis. The notation $a=(a_1,\dots, a_n)$ also refers to $a$ as a column vector with $i$th component $a_i$.
			\item For an $n\times n$ matrix $A$ and subsets $I,J\subset [n]$, the notation $A[I|J]$ denotes the submatrix of $A$ obtained by picking the rows from $I$ and the columns from $J$. We write $A[I]$ for the principal submatrix $A[I|I]$. 
			\item In the case where $I=[k]$, we simply write $A[k|I]$ or $A[I|k]$ for $A[[k]|I]$ and $A[I|[k]]$ respectively.
			\item If $u$ is a column vector, we write $u[I]$ for $u[I|1]$. Similarly, if $v$ is a row vector, then $v[I]=v[1|I]$.
			\item We assume that $\CCC^n$ is endowed with the inner product $\langle u,v\rangle=u^*v$, where $A^*$ is the conjugate transpose of a matrix $A$.
			\item For $I\subset [n]$, we let $\CCC^I=\langle e_i:i\in I\rangle\subset \CCC^n$. We identify $\CCC^I$ with the set $\{(z_i)_{i\in I}:z_i\in \CCC\}$ when no confusions arise.\\
		\end{enumerate}
		
		\noindent\textbf{Acknowledgments.} The author is very grateful with Abeer Al-Ahmadieh, Jim Geelen, Elana Kalashnikov and Ruxandra Moraru for helpful conversations.
    
\section{Background and Motivation}\label{sec:background}
	The problem studied in this paper is related to the broader problem of determining a matrix by its principal minors. In general, principal minors do not determine a matrix uniquely. For example, any two matrices that are related by diagonal similarity or transposition must have the same principal minors. The main theorem of Loewy in \cite{Loewy1986} states that, for $n\ge 4$, if two $n\times n$ matrices $A$ and $B$ that have the same corresponding principal minors of all orders, with $A$ irreducible and such that either $\rank A[I|I^c]\ge 2$ or $\rank A[I^c|I]\ge 2$ for all $I\subset [n]$ with $2\le |I|\le n-2$, then $B=DAD^{-1}$ or $B=DA^tD^{-1}$ for $D$ an invertible diagonal matrix. Here, $A^t$ denotes the transpose of $A$ and $I^c\coloneqq [n]-I$. These two assumptions, which we now recall, are conditions on the rank of certain submatrices of $A$ and are related to the notion of $k$-connectivity of matroids. 
	First, recall that the matrix $A$ is \textbf{irreducible} if there is no permutation matrix $P$ such that $PAP^{-1}$ is block upper triangular. The second condition in this theorem says that $A$ has no ``cuts".
	
	\begin{Def}
		A \textbf{cut} of an $n\times n$ matrix $A$ is a subset $I\subset [n]$ with $2\le |I|\le n-2$ such that both $\rank A[I|I^c]$ and $\rank A[I^c|I]$ are at most 1.
	\end{Def}
	
	These rank conditions in the case of orthogonal projection matrices appear in the theory of matroids. For us, a \textbf{matroid} will only refer to a pair $M_W=([n],r_W)$ for $W$ a subspace of $\CCC^n$ and $r_W$ a \textbf{rank function} on $[n]$ given by $r_W(I)=\dim \pi_I(W)$, where $\pi_I:\CCC^n\to \CCC^I$ is the orthogonal projection. These are more commonly known as representable matroids over $\CCC$ \cite{Oxley2011}.
	The irreducibility and cut conditions can be expressed in terms of this rank function, and give rise to the connectivity function which encodes the rank of $A[I|I^c]$ (Proposition \ref{prop:principal-minor-rank-is-connectivity-func}).
	
	The wider problem of determining a matrix by its principal minors was considered in \cite{EngelSchneider1980} for the case of symmetric matrices and since then has been extended to more cases \cite{HartfielLoewy1984, Loewy1986, AlAhmadieh2026}. More recently, Chatterjee–Ghosh–Gurjar–Raj \cite{10.1145/3717823.3718146} characterized principal-minor equivalence for irreducible matrices using cut-transpose operations and diagonal equivalence.
	Another notable problem in this direction is that of the principal minor assignment problem \cite{ahmadieh2021characterizingprincipalminorssymmetric, GriffinTsatsomeros2006}.

    The set $|W|$, with $W\subset \CCC^n$ a subspace, appears in many areas of mathematics, such as matroid theory as a subset of the $\triangle$-vectors associated to matroids over the triangular hyperfield $\triangle$ \cite{Anderson2019, BakerBowler2019}, and in symplectic geometry as the (square-rooted) image of subvarieties under a toric moment map \cite{MikhalkinAmoebas}. This last example is our main motivation for this paper and will be explored in an upcoming paper.
    
    More precisely, for a subspace $W$ of dimension $r$, consider the set $\II_W$ of all subsets $I\subset [n]$ such that $\dim(W^\perp\cap \CCC^I)=1$. For a unit vector generating $W^\perp\cap \CCC^I$,
    let $p_I^W$ be its image under the coordinatewise absolute value map. One can show that these points encode the data of the principal minors of the orthogonal projection matrices onto $W^\perp$. Further, if the orthogonal projection matrices onto $W_1$ and $W_2$ have the same principal minors, then so do the orthogonal projection matrices onto $W_1^\perp$ and $W_2^\perp$. Combining this with Theorem \ref{thm:same-principal-minors-have-equal-image} gives us the following result, whose details we defer to a subsequent paper.
    
    \begin{thm*}
        If two $r$-dimensional subspaces $W_1,W_2\subset \CCC^n$ are such that $\II_{W_1}=\II_{W_2}\eqqcolon \II$ and $p_I^{W_1}=p^{W_2}_{I}$ for all $I\in \II$, then $|W_1^\perp|=|W_2^\perp|$.
        \label{thm:same-special-points-have-equal-image}
    \end{thm*}
    
    Note that the principal minors of an orthogonal projection matrix $A$ and the normalized Pl\"ucker coordinates $p_B(W)$ of $W=\im A$, where $\dim W=r$ and $B\in \binom{[n]}{r}$, are related by
    \[
        \det A[I]=\sum_{\substack{B\in \binom{[n]}{r} \\ I\subset B}} |p_B(W)|^2,
    \]
    with the sum understood to be zero if $|I|>r$. These coordinates are normalized such that $\sum_{|B|=r}|p_B(W)|^2=1$.
    Hence, Theorem \ref{thm:same-principal-minors-have-equal-image} can be interpreted as saying that the absolute values of the Pl\"ucker coordinates of $W$ determine $|W|$. As mentioned before, the theorem of Loewy gives us this with ease under the right assumptions. The main obstruction in extending to a general orthogonal projection matrix is the existence of off-diagonal blocks with small rank, that is, cuts. However, for $A$ irreducible, cuts give us a preferred direction in $W$ and hence acquire a more geometric and combinatorial meaning. This lets us construct $|W|$ from smaller dimensional subspaces and hence allows for induction.

\subsection{Outline of the Paper}
	We begin in Section \ref{sec:Results-from-Lin-Alg} by developing the linear algebra needed for the proof of Theorem \ref{thm:same-principal-minors-have-equal-image}. The main difficulty occurs when the orthogonal projection matrix is irreducible but has a cut. For such a projection $A$ with $W=\im A$ and $I\subset [n]$ a cut, we show that $A$ admits a block decomposition determined by the subspaces
	\[
		E=W\cap \CCC^I,\qquad F=W\cap \CCC^{I^c},
	\]
	together with unit vectors $u\in E^\perp\cap \CCC^I$, $v\in F^\perp\cap \CCC^{I^c}$ and a number $\alpha\in(0,1)$. In particular,
	\[
		W=E\oplus F\oplus
		\left\langle \sqrt{\alpha}u+\sqrt{1-\alpha}v\right\rangle.
	\]

	This allows us to describe both $W$ and $|W|$ as a gluing of two subspaces inside smaller coordinate spaces and hence points to a proof of Theorem \ref{thm:same-principal-minors-have-equal-image} by induction. We study the determinantal polynomial of an orthogonal projection having this form. We show that the principal minors of the projections onto the smaller subspaces $E$ and $F$ are related to those of $A$. We then prove an extension of Loewy's theorem to the cases $n=1,2,3$ which we will need for our induction, and show that, for $n\ge 4$, the diagonal matrix in Loewy's theorem can be taken to be unitary. We conclude with some basic results concerning simultaneous permutations of rows and columns.

	Section \ref{sec:proof-of-thm-11} contains the proof of Theorem \ref{thm:same-principal-minors-have-equal-image}. As mentioned above, the proof is by induction on the size $n$ of the projection matrices. The cases $n=1,2,3$ follow from the low-dimensional results of the preceding section. For $n\ge 4$, we divide the argument into three cases according to the structure of the orthogonal projection matrix $A$.

	If $A$ is reducible, then, after a simultaneous permutation of rows and columns, it decomposes into smaller orthogonal projection matrices. Equality of principal minors forces the corresponding decomposition for $B$, and the result follows immediately from the induction hypothesis. If $A$ is irreducible and has no cuts, then $A$ and $B$ differ, up to transposition, by conjugation by a diagonal unitary matrix. Since coordinatewise absolute value is invariant under diagonal unitary transformations and complex conjugation, the desired equality $|\im A|=|\im B|$ follows.

	The remaining case with $A$ irreducible and having a cut is the main part of the proof. We first show that the same subset $I$ is also a cut of $B$, so both matrices admit the structure theorem developed in Section \ref{sec:Results-from-Lin-Alg}. This produces two smaller spaces associated to each of $A$ and $B$. 
	Using the determinantal polynomial formulas, we show that the orthogonal projections onto the corresponding smaller spaces have equal principal minors. The induction hypothesis therefore identifies their images under coordinatewise absolute value. The gluing description of $|W|$ then shows that these equalities combine to give
	\[
		|\im A|=|\im B|,
	\]
	completing the proof of the main theorem.

	The final two sections connect the preceding decomposition into a more general geometric and combinatorial framework. In Section \ref{sec:stratification}, we introduce a connectivity function
	\[
		\lambda_W(I)
		=\dim W-\dim(W\cap\CCC^I)-\dim(W\cap\CCC^{I^c})
	\]
	for a subspace $W\subset\CCC^n$, which comes from matroid theory. We show that, if $A$ is the orthogonal projection onto $W$, then
	\[
		\lambda_W(I)=\rank A[I|I^c].
	\]
	
	Thus the three cases appearing in the proof of the main theorem can be interpreted in terms of connectivity: reducible projections correspond to reducible subspaces, irreducible projections with a cut correspond to exactly $2$-connected subspaces, and irreducible projections with no cuts correspond to $3$-connected subspaces. More generally, the connectivity level gives a finite stratification of the Grassmannian $\Gr(r,n)$. We study this stratification throughout Section \ref{sec:stratification} and end with a computation of the open stratum using results from matroid. Some of these results are proven in a self-contained manner in Appendix \ref{sec:structure-of-k-conn-subspaces}. In this appendix, we also prove that a generalization of the structure theorem from Section \ref{sec:structure-ofo-exactly-2-conn} to exactly $k$-connected subspaces can be obtained from the cosine-sine decomposition.

\section{Some Results from Linear Algebra}\label{sec:Results-from-Lin-Alg}
	The main goal of this section is to set up some fundamental results needed in the proof of Theorem \ref{thm:same-principal-minors-have-equal-image}. We first give a decomposition of irreducible orthogonal projection matrices that have a cut and of their images in terms of smaller orthogonal projection matrices. We use this decomposition to relate the principal minors of such a projection with the smaller pieces obtained from the decomposition. The remainder of the section gives an extension of Loewy's theorem to lower dimensional cases, a corollary of Loewy's theorem for the case of orthogonal projection matrices and some linear algebra results that show that we can simultaneously permute rows and columns without changing the equality of principal minors between two matrices.
	
\subsection{The Structure of Irreducible Orthogonal Projections with a Cut}\label{sec:structure-ofo-exactly-2-conn}
	Let $A$ be an $n\times n$ irreducible orthogonal projection matrix, $n\ge 4$, and let $I$ be a cut of $A$. Write $W$ for the image of $A$. By irreducibility, we must have $\rank A[I|I^c]=1$. Hence, we can write
	\[
		A[I|I^c]=suv^*
	\]
	for some unit vectors $u\in \CCC^{I}$, $v\in \CCC^{I^c}$, and $s>0$. Since $A$ is Hermitian and $A^2=A$, we must have $A[I|I^c]^*=A[I^c|I]$ so
	\begin{align}
		A[I] &=A[I]^2+s^2uu^*,\label{eq:A[I]-for-cuts}\\
		A[I^c] &= A[I^c]^2+s^2vv^*,\label{eq:A[I^c]-for-cuts}\\
		uv^* &=A[I]uv^*+uv^*A[I^c]. \label{eq:A[I|I^c]-for-cuts}
	\end{align}
	Multiplying equation (\ref{eq:A[I|I^c]-for-cuts}) on the left by $I-uu^*$ and on the right by $v$ gives
	\[
		A[I]u=(u^*A[I]u)u,
	\]
	so that $u$ is an eigenvector of $A[I]$ with eigenvalue $\alpha\coloneqq u^*A[I]u$.
	If we instead had multiplied (\ref{eq:A[I|I^c]-for-cuts}) by $I-vv^*$ on the right and by $u^*$ on the left then we would have gotten
	\[
		A[I^c]v=(v^*A[I^c]v)v,
	\]
	so $v$ is an eigenvector of $A[I^c]$ with eigenvalue $\beta\coloneqq v^*A[I^c]v$.
	
	Using equation (\ref{eq:A[I|I^c]-for-cuts}) again, we find that
	\[
		1=\alpha+\beta.
	\]
	Now, by multiplying equation (\ref{eq:A[I]-for-cuts}) by $u$ on the right, we find $\alpha=\alpha^2+s^2$ and hence
	\[
		s^2=\alpha-\alpha^2=\alpha(1-\alpha).
	\]
	Since $s>0$, we have $0<\alpha<1$.
	
	Equation (\ref{eq:A[I]-for-cuts}) above says that the restriction of $A[I]$ to $u^\perp$ is a projection, and since $A[I]$ is Hermitian, it is an orthogonal projection when restricted to $u^\perp$, where the orthogonal complement is taken inside $\CCC^{I}$. By decomposing $\CCC^{I}$ as $u^\perp\oplus \langle u\rangle$, we have that
	\[
		A[I]=P_E+\alpha uu^*,
	\]
	where $P_E=A[I]|_{u^\perp}$ and $E=\im P_E$. Note that $E=W\cap \CCC^I$. Indeed, if $z\in u^\perp$, then
	\[
		Az=\begin{pmatrix}
			A[I]z\\
			A[I^c|I]z
		\end{pmatrix}=\begin{pmatrix}
			P_Ez\\
			0
		\end{pmatrix}.
	\]
	The left hand side is in $W$ and the right hand side is in $\CCC^I$, so $P_Ez\in W\cap \CCC^I$. The converse is similar.
	
	We can repeat the same procedure with $I^c$ and find that
	\[
		A[I^c]=P_F+\beta vv^*,
	\]
	for $P_F=A[I^c]|_{v^\perp}$ an orthogonal projection matrix with image $F=W\cap \CCC^{I^c}$.
	
	The work above allows us to have a full decomposition of $W$ and its orthogonal projection $A$.
	
	\begin{thm}
		Let $A$ be an $n\times n$ orthogonal projection matrix which is irreducible and has a cut $I$. Write $W$ for the image of $A$. Then $A$ has a block decomposition in terms of the orthogonal projections $P_E$ and $P_F$ onto $E=W\cap \CCC^I$ and $F=W\cap \CCC^{I^c}$ respectively, as
		\begin{enumerate}
			\item $A[I]=P_E+\alpha uu^*$,
			\item $A[I^c]=P_F+(1-\alpha) vv^*$,
			\item $A[I|I^c]=\sqrt{\alpha(1-\alpha)}uv^*$,
		\end{enumerate}
		where $u\in \CCC^I\cap E^\perp$, $v\in \CCC^{I^c}\cap F^\perp$ are unit vectors and $\alpha\in (0,1)$.
		These have the additional property that
		\begin{enumerate}
			\item [(4)] $\alpha$ is the only eigenvalue of $A[I]$ in $(0,1)$, counted with multiplicity,
			\item [(5)] $1-\alpha$ is the only eigenvalue of $A[I^c]$ in $(0,1)$, counted with multiplicity,
			\item [(6)] all other eigenvalues of $A[I]$ and $A[I^c]$ are either $0$ or $1$.
		\end{enumerate}
		This induces a decomposition of $W$ as
		\[
			W=E\oplus F\oplus \langle \sqrt{\alpha}u+\sqrt{\beta}v\rangle,\quad \beta\coloneqq 1-\alpha.
		\]		
		\label{thm:cut-decomposition}
	\end{thm}
	\vspace{-0.9cm}
	\begin{proof}
		It only remains to show the additional properties (4)-(6). First, since $A[I]$ is a principal submatrix of a projection matrix, its eigenvalues must all lie in $[0,1]$.
		
		Write $C=A[I|I^c]$. Then $CC^*=A[I]-A[I]^2$ so that the eigenvalues of $CC^*$ are exactly of the form $\lambda-\lambda^2$ with $\lambda$ ranging through the spectrum of $A[I]$. Since $\rank CC^*=1$, this means that there is only one $\lambda$, with multiplicity, in the spectrum of $A[I]$ such that $\lambda(1-\lambda)\neq 0$, that is, all but one eigenvalue of $A[I]$ are equal to $0$ or $1$. Since $\alpha$ is such an eigenvalue, it must be the only one.
		
		The proof for $A[I^c]$ is similar.
	\end{proof}

	The second part of the theorem has a very geometric description. Consider the vector spaces $\CCC \times \CCC^I$ and $\CCC \times \CCC^{I^c}$ where the first component of each is generated by new coordinate vectors $e_0$ and $e_0'$ respectively. 
	Extend the induced inner products on $\CCC^I$ and $\CCC^{I^c}$ from $\CCC^n$ to the products above in such a way that $e_0$ is a unit vector orthogonal to $\CCC^I$, and similarly for $e_0'$ and $\CCC^{I^c}$.
	Then the theorem above then says $W$ is constructed as follows: pick subspaces $E\subset \CCC^I$, $F\subset \CCC^{I^c}$ and unit vectors $u\in \CCC^I\cap E^\perp$, $v\in \CCC^{I^c}\cap F^\perp$ together with $\alpha\in (0,1)$, and consider
	\[
		\EE\coloneq E\oplus \langle e_0+\sqrt{\alpha}u\rangle,\qquad  \FF\coloneqq F\oplus \langle e_0'+\sqrt{1-\alpha}v\rangle
	\]
	together with the projections 
	\begin{alignat*}{2}
		\pi_0 &:\EE\to \CCC, \qquad\qquad \pi_0(e+\lambda(e_0+\sqrt{\alpha}u)) &&=\lambda,\\
		\pi_0'&:\FF\to \CCC, \qquad \pi_0'(f+\lambda(e_0'+\sqrt{1-\alpha}v)) &&=\lambda, 
	\end{alignat*}
	where $e\in E$, $f\in F$.
	Then $W$ is the projection of
	\[
		\EE\times_\CCC \FF =\Big\{(e,f)\in \EE\times \FF: \pi_0(e)=\pi_0'(f)\Big\}
	\]
	under the projection to $\CCC^I\times \CCC^{I^c}$. The following corollary then says that the absolute value map respects this decomposition, i.e., that $|W|$ is the projection of $|\EE|\times_{\RRR_{\ge 0}}|\FF|$ onto $\RRR^I_{\ge 0}\times \RRR^{I^c}_{\ge 0}$.
	
	\begin{cor}
    	Let $W$ be the image of an $n\times n$ orthogonal projection matrix which is irreducible and has a cut $I$. Write
    	\[
    		W=E\oplus F\oplus \langle \sqrt{\alpha}u+\sqrt{\beta}v\rangle
    	\]
    	as in Theorem \ref{thm:cut-decomposition}. Consider new coordinate vectors $e_0$ and $e_0'$ and let
    	\begin{align*}
    		\EE &\coloneqq \langle e_0+\sqrt{\alpha}u\rangle\oplus E \subset \CCC e_0\oplus \CCC^I\\
    		 \FF &\coloneqq \langle e'_0+\sqrt{\beta}v\rangle\oplus F \subset \CCC e_0'\oplus \CCC^{I^c}.
    	\end{align*} 
    	Then
    	\[
    		|W|=\left\{(a,b)\in \RRR_{\ge 0}^I\oplus \RRR_{\ge 0}^{I^c}:\text{there exists }c\in \RRR_{\ge 0}\text{ with } (c,a)\in |\EE|,\ (c,b)\in |\FF|\right\}
    	\]
    	\label{cor:absolute-image-of-2-connected-subspace}
    \end{cor}
    \begin{proof}
    	If $z\in W$, then $z=e+f+\lambda (\sqrt{\alpha}u+\sqrt{\beta}v)$ for some $e\in E$, $f\in F$ and $\lambda\in \CCC$. Let $a=|e+\lambda \sqrt{\alpha}u|$, $b=|f+\lambda \sqrt{\beta}v|$ and $c=|\lambda|$ so that $|z|=(a,b)\in \RRR^I\oplus \RRR^{I^c}$ and
    	\begin{align*}
    		(c,a) &=(|\lambda|,|e+\lambda\sqrt{\alpha}u|)=|e+\lambda(\sqrt{\alpha}u+e_0)|\in |\EE|,\\
    		(c,b) &=|\lambda e_0'|+|f+\lambda\sqrt{\beta}v|=|f+\lambda(\sqrt{\beta}v+e_0')|\in |\FF|.
    	\end{align*}
    	
    	Conversely, if $(a,b)\in \RRR_{\ge 0}^I\oplus \RRR^{I^c}_{\ge 0}$ is such that there is some $c\in \RRR_{\ge 0}$ with $(c,a)\in |\EE|,\ (c,b)\in |\FF|$, then we can write
    	\[
    		a=|e+\lambda_1 \sqrt{\alpha}u|,\qquad b=|f+\lambda_2\sqrt{\beta}v|
    	\]	
    	for some $e\in E,f\in F, \lambda_1,\lambda_2\in \CCC$ with $|\lambda_1|=|\lambda_2|=c$. If $c=0$ we are done, and if $c\neq 0$, then for $t=\lambda_1/\lambda_2\in S^1$,
    	\[
    		b=|f+\lambda_2\sqrt{\beta}v|=|t||f+\lambda_2\sqrt{\beta}v|=
    		 \left| tf+\lambda_1 \sqrt{\beta}v \right|.
    	\]
    	Let $\lambda=\lambda_1$ and $z=e+tf+\lambda(\sqrt{\alpha}u+\sqrt{\beta}v)$. Then $z\in W$ with $|z|=(a,b)$.
    \end{proof}

\subsection{Determinantal Polynomials and Orthogonal Projections}\label{sec:determinantal-of-exactly-2-conn}
	In order to study the principal minors of a matrix, we need to use a polynomial that keeps track of the principal minors of a matrix. A natural choice for this is the determinantal polynomial of a matrix $A$, which is a polynomial whose coefficients are the principal minors of $A$:
	
	\begin{Def}
        For an $n\times n$ matrix $A$, its \textbf{determinantal polynomial} is
        \[
            f_A=\det(A+\diag(x_1,\dots, x_n)).
        \]
    \end{Def}
    
    Using the Leibniz determinant formula we obtain the following.
    
    \begin{prop}
        The coefficients of $f_A$ are the principal minors of $A$:
        \[
            f_A=\sum_{I\subset [n]}\det A[I]\prod_{j\not\in I}x_j,
        \]
        where we set $\det A[\emptyset]=1$.
        In particular, two matrices $A$ and $B$ have the same corresponding principal minors of all orders if and only if $f_A=f_B$.
        \label{prop:coefficients-of-determinantal-poly-are-minors}
    \end{prop}
    
    When the matrix $A$ has the decomposition given by Theorem \ref{thm:cut-decomposition} with $I=[k]$, i.e.,
    \[
		A=\begin{pmatrix}
			P_E+\alpha uu^* & \sqrt{\alpha\beta}uv^*\\
			\sqrt{\alpha\beta}vu^* & P_F+\beta vv^*
		\end{pmatrix}, \qquad \alpha+\beta=1,
	\]
	we obtain a natural family of subspaces $E$ and $F$ that live in the smaller coordinate subspaces $\CCC^I$ and $\CCC^{I^c}$ respectively. Hence, in order to do induction, we need to compare the principal minors of $A$ with those of $P_E$ and $P_F$.
	The defect between their corresponding determinantal polynomials is captured by the following polynomial:
	
	\begin{Def}
		For a square matrix $M$ indexed by $S\subset [n]$, define 
		\[
			a_{M,u}\coloneqq u^*\mathrm{adj}(\diag(x_i:i\in S)+M)u.
		\]
		where $\mathrm{adj}(A)$ is the adjugate matrix of $A$, that is, the matrix whose $(s_i,s_j)$-entry is $(-1)^{i+j}\det A[S\setminus \{s_j\}|S\setminus \{s_i\}]$ after writing $S=\{s_1<\dots<s_m\}$.
	\end{Def}
	
	The reason the adjugate needs more care with the indices is that the positions of entries in a submatrix of $A$ need not match their positions inside of $A$ itself.
	With this in mind, the main objective of this section is to prove the following result:
	\begin{thm}
		For an orthogonal projection matrix 
		\[
			A=\begin{pmatrix}
				P_E+\alpha uu^* & \sqrt{\alpha\beta}uv^*\\
				\sqrt{\alpha\beta}vu^* & P_F+\beta vv^*
			\end{pmatrix}, \qquad \alpha+\beta=1,
		\]
		satisfying Theorem \ref{thm:cut-decomposition}, its determinantal polynomial is given by
		$$f_A =f_{P_E+\alpha uu^*}\cdot f_{P_F+\beta vv^*}-\alpha\beta a_{P_E,u}\cdot a_{P_F,v}.$$
		\label{thm:determinantal-poly-of-full-matrix-with-cut}
	\end{thm}
	\vspace{-0.9cm}

    \begin{rem}
    	Since we will later be comparing matrices of different size, we establish the following convention: if $A$ is a matrix indexed by a set $S\subset[n]$, we define
	 	\[
	 		D_S(x)\coloneqq \diag(x_s:s\in S)
    	\] 
    	and
    	\[
    		f_A\coloneqq \det(A+D_S(x))\in \CCC[x_s:s\in S].
    	\]
    	When $S=[n]$, we simply write $D_{[n]}(x)=D(x)$.
    \end{rem}
   
  	By taking $x_1=\dots=x_n=t$ and replacing $A$ by $-A$ in Proposition \ref{prop:coefficients-of-determinantal-poly-are-minors}, we obtain the following well-known result.
    
    \begin{cor}
    	The coefficients of the characteristic polynomial $\lambda_A(t)$ of an $n\times n$ matrix $A$ correspond to the sum of the principal minors:
    	\[
    		\lambda_A(t)\coloneqq \det(tI-A)=t^n+c_1t^{n-1}+\dots +c_n,\quad c_i=(-1)^i\sum_{\substack{I\subset [n],|I|=i}}\det A[I]
    	\]
        where the sum on the right is over all subsets of $[n]$ of cardinality $i$.
    	\label{cor:coeff-of-char-poly}
    \end{cor}

    For the remainder, we focus on finding a relation between $f_A$, $f_{P_E}$ and $f_{P_F}$.
    
    \begin{lem}
    	For any square matrix $A$ indexed by a set $S\subset [n]$ and $u\in \CCC^S$,
    	\[
    		f_{A+\alpha uu^*}=f_A+\alpha a_{A,u}.
    	\]
    	\label{prop:determinantal-poly-of-cut}
    \end{lem}
    \vspace{-0.8cm}
    \begin{proof}
    	For any square matrix $M$ whose columns are $M_1,\dots, M_n$, we can use multilinearity of the determinant to have
    	\begin{align*}
    		\det(M+\alpha uu^*) &=\det(M_1+\alpha \bar{u}_1u,\dots, M_n+\alpha\bar{u}_n u)\\
    		&=\det(M)+\alpha \sum_{j=1}^n \bar{u}_j\det(M_1,\dots, M_{j-1},u,M_{j+1},\dots, M_n),
    	\end{align*}
    	where the terms with more than one $u$ vanish.
    	For each $j$, do a cofactor expansion along the $j$th column to have
    	\begin{align*}
    		\det(M+\alpha uu^*) &=\det(M)+\alpha \sum_{i,j=1}^n (-1)^{i+j}\bar{u}_ju_i\det M[\{i\}^c|\{j\}^c]\\
    		&=\det(M)+\alpha \sum_{i,j=1}^n \bar{u}_ju_i\mathrm{adj}(M)_{j,i} \\
    		&=\det(M)+\alpha u^* \mathrm{adj}(M)u.
    	\end{align*}
    	Taking $M=A+D_S(x)$ concludes the proof.
    \end{proof}
    
    Let us compute the determinantal polynomial of
    \[
		A=\begin{pmatrix}
			P_E+\alpha uu^* & \sqrt{\alpha\beta}uv^*\\
			\sqrt{\alpha\beta}vu^* & P_F+\beta vv^*
		\end{pmatrix}.
	\]
	We have
	\begin{align*}
		f_A &= \det(D(x)+A)\\
		&=\det\begin{pmatrix}
    			\displaystyle D_I(x) + P_E+\alpha uu^* & \displaystyle\sqrt{\alpha\beta}uv^*\\
    			\displaystyle \sqrt{\alpha\beta}vu^* & \displaystyle D_{I^c}(x)+P_F+\beta vv^*
    		\end{pmatrix}.
	\end{align*}
	Over the set of all $x$ such that the diagonal blocks are invertible, we can apply the Schur determinant formula to have
	\begin{align*}
		f_A &=\det(D_I(x) + P_E+\alpha uu^*)\det(D_{I^c}(x)+P_F+\beta vv^*)\\
		&\cdot \det(I-{\alpha\beta}vu^*(D_I(x) + P_E+\alpha uu^*)^{-1}uv^*(D_{I^c}(x)+P_F+\beta vv^*)^{-1}).
	\end{align*}
	Then Sylvester's determinant identity allows us to write this last factor as
	\[
		1-\alpha \beta \Big(v^*(D_{I^c}(x)+P_F+\beta vv^*)^{-1}v\Big)\Big( u^*(D_I(x) + P_E+\alpha uu^*)^{-1}u\Big)
	\]
	Since $M^{-1}=\det(M)^{-1}\mathrm{adj}(M)$ for any invertible matrix $M$, we have
	\begin{align*}
		(D_I(x) + P_E+\alpha uu^*)^{-1} &=\frac{1}{\det(D_I(x) + P_E+\alpha uu^*)}\mathrm{adj}(D_I(x) + P_E+\alpha uu^*)\\
		&=\frac{1}{f_{P_E+\alpha uu^*}}\mathrm{adj}(D_I(x) + P_E+\alpha uu^*).
	\end{align*}
	Continuing similarly for the other factors, we have
	\begin{align*}
		f_A 
		&=f_{P_E+\alpha uu^*}\cdot f_{P_F+\beta vv^*}-\alpha\beta a_{P_E+\alpha uu^*,u}\cdot a_{P_F+\beta vv^*,v}.
	\end{align*}
	This equation holds over all $x$ in the dense open set where $D_I(x) + P_E+\alpha uu^*$ and $D_{I^c}(x)+P_F+\beta vv^*$ are invertible. Hence, the equality above holds for all $x$.
	
	Theorem \ref{thm:determinantal-poly-of-full-matrix-with-cut} is then obtained from the following result.
	
	\begin{lem}
		For a square matrix $M$ indexed by $S\subset [n]$, we have $a_{M+\alpha uu^*,u}=a_{M,u}$.
	\end{lem}
	\begin{proof}
		If $\alpha=0$ there is nothing to prove so take $\alpha\neq 0$.
		Using Lemma \ref{prop:determinantal-poly-of-cut} with $A=M+\alpha uu^*$ gives
		\begin{align*}
			\alpha a_{M+\alpha uu^*,u} &=f_{M+2\alpha uu^*}-f_{M+\alpha uu^*}\\
			&=f_M+2\alpha u^*\mathrm{adj}(D_S(x)+M)u-\big( f_M+\alpha u^* \mathrm{adj}(D_S(x)+M)u\big)\\
			&=\alpha a_{M,u}
		\end{align*}
		so $a_{M+\alpha uu^*,u}=a_{M,u}$.
	\end{proof}

\subsection{Low-Dimensional Cases}
	Loewy's theorem can be extended to $n=1,2,3$ by restricting to Hermitian matrices. 
	
	\begin{prop}
		If $n=1,2,3$ and $A,B$ are $n\times n$ Hermitian matrices that have the same corresponding principal minors, then there is a diagonal unitary matrix $T$ such that either $B=TAT^{-1}$ or $B^t=TAT^{-1}$.
		\label{prop:Loewy-for-low-n-orthogonal}
	\end{prop}
	
	To prove this, we first note the following result which we will be constantly using without appeal.

	\begin{lem}
        Suppose $A,B$ are Hermitian matrices having the same corresponding principal minors of all orders. Then,
        \begin{enumerate}
            \item for all $i$, $a_{ii}=b_{ii}$, and
            \item $|a_{ij}|=|b_{ij}|$ for all $i\neq j$.
        \end{enumerate}
        In particular, $a_{ij}=0$ if and only if $b_{ij}=0$, so that $A$ is irreducible if and only if $B$ is irreducible.
        \label{lem:components-of-hermitian-with-same-minors}
    \end{lem}
    \begin{proof}
        We have that $a_{ii}=\det A[\{i\}]=\det B[\{i\}]=b_{ii}$, and
        \[
            a_{ii}a_{jj}-|a_{ij}|^2=\det A[\{i,j\}]=\det B[\{i,j\}]=b_{ii}b_{jj}-|b_{ij}|^2.
        \]
        Hence, $|a_{ij}|=|b_{ij}|$ which concludes the proof.
    \end{proof}

	\begin{proof}[Proof of Proposition \ref{prop:Loewy-for-low-n-orthogonal}]
	 	For $n=1$, the proof is trivial of course.
	
		Consider the case where $n=2$ and let $A,B$ be two $2\times 2$ Hermitian matrices that have the same principal minors. Writing $a_{ij}$, $b_{ij}$ for the $(i,j)$-entry of these matrices, Lemma \ref{lem:components-of-hermitian-with-same-minors}  gives us that $|a_{12}|=|b_{12}|$. Hence, there is some $t\in S^1$ such that $b_{12}=ta_{12}$ and so
		\[
			B=\begin{pmatrix}
				b_{11} & b_{12}\\
				\bar{b}_{12} & b_{22}
			\end{pmatrix}=\begin{pmatrix}
				1 & 0\\
				0 & t^{-1}
			\end{pmatrix}\begin{pmatrix}
				a_{11} & a_{12}\\
				\bar{a}_{12} & a_{22}
			\end{pmatrix}\begin{pmatrix}
				1 & 0\\
				0 & t
			\end{pmatrix}=T^{-1}AT.
		\]
		
		Finally, for $n=3$, we can use Lemma \ref{lem:components-of-hermitian-with-same-minors} to write the determinant of the Hermitian matrix $A$:
		\[
			\det A=a_{11}a_{22}a_{33}+a_{12}a_{23}\bar{a}_{13}+\bar{a}_{12}a_{13}\bar{a}_{23}-a_{11}|a_{23}|^2-a_{22}|a_{13}|^2-a_{33}|a_{12}|^2.
		\]
		A similar formula holds for $\det B$ and by comparing the coefficients, we have that
		\[
			b_{12}b_{23}\bar{b}_{13}+\bar{b}_{12}b_{13}\bar{b}_{23}=a_{12}a_{23}\bar{a}_{13}+\bar{a}_{12}a_{13}\bar{a}_{23}
		\]
		or equivalently,
		\[
			\Re(b_{12}b_{23}\bar{b}_{13})=\Re(a_{12}a_{23}\bar{a}_{13}).
		\]
		There are three cases to consider depending on the number of vanishing off-diagonal elements.
		
		\emph{Case 1: $A$ is reducible.} Equivalently, at least two $a_{ij},a_{jk}$ vanish, for $i,j,k$ distinct. In this case, up to a simultaneous permutation of rows and columns, the matrix $A$ is block diagonal with blocks of size at most $2\times 2$. By the preceding case, the result holds.
		
		\emph{Case 2: Only one $a_{ij}=0$.} Up to a simultaneous permutation of the rows and columns, we can assume that $a_{13}=0$. Then let $t_{2}=b_{12}\bar{a}_{12}/|a_{12}b_{12}|$ and $t_{3}=b_{23}\bar{a}_{23}/|a_{23}b_{23}|$ so that
		\[
			t_{2}a_{12}=\frac{b_{12}}{|b_{12}|}|a_{12}|=b_{12}, \qquad t_{3}a_{23}=\frac{b_{23}}{|b_{23}|}|a_{23}|=b_{23}.
		\]
		Then, for $T=\diag(1,t_2,t_2t_3)$, we have $B=T^{-1}AT$.
		
		\emph{Case 3: All $a_{12},a_{23},a_{13}$ are nonzero.} By considering $D^{-1}AD$ with $D$ the diagonal matrix $D=\diag(1,\bar{a}_{12}/|a_{12}|,\bar{a}_{23}\bar{a}_{12}/|a_{23}a_{12}|)$, we can assume that $a_{12},a_{23}>0$. Similarly, we can conjugate $B$ to $\tilde{D}^{-1}B\tilde{D}$ so that  $b_{12},b_{23}>0$.  Then $\tilde{D}^{-1}B\tilde{D}$ and $D^{-1}AD$ have the same principal minors, so that $a_{12}=b_{12}, a_{23}=b_{23}$ and, after writing $a_{13}=re^{i\theta}$, $b_{13}=re^{i\theta'}$, the equality of their determinants gives $\cos(\theta)=\cos(\theta')$. Therefore, $\theta'=\pm\theta$ mod $2\pi$. In the first case, $b_{13}=a_{13}$ so
		\[
			B=(D\tilde{D}^{-1})^{-1}A(D\tilde{D}^{-1}).
		\]
		In the second case, $b_{13}=\bar{a}_{13}$ so that
		\[
			B=\tilde{D}\bar{D^{-1}AD}\tilde{D}^{-1}=(\tilde{D}{D}) A^t\bar{(\tilde{D}{D})}.
		\]
		Hence, for $T=(\tilde{D}D)^{-1}$, we have $B^t=T A T^{-1}$.
	 \end{proof}

\subsection{Higher Dimensional Cases}
	In the case of $n\ge 4$ and $A$ and $B$ orthogonal projection matrices, the diagonal matrix in Loewy's theorem can be taken to be unitary.
	
	\begin{prop}
        Suppose $A,B$ are $n\times n$ Hermitian matrices, $n\ge 4$, which have the same corresponding principal minors of all orders. Suppose also that $A$ is irreducible and has no cuts.
        Then there exists a diagonal unitary matrix $T$ such that either $B=TAT^{-1}$ or $B^t=TAT^{-1}.$
        \label{prop:diagonals-in-Loewy-are-unitary}
    \end{prop}
    \begin{proof}
        The assumptions guarantee that Loewy's theorem can be applied, so there exists an invertible diagonal matrix $D$ such that either $B=DAD^{-1}$ or $B^t=DAD^{-1}$. In the first case, for $i,j$ distinct with $a_{ij}\neq 0$,
        \[
            b_{ij}=\frac{d_i}{d_j}a_{ij}.
        \]
        Taking absolute values and using Lemma \ref{lem:components-of-hermitian-with-same-minors},
        \[
            |a_{ij}|=|b_{ij}|=\frac{|d_i|}{|d_j|}|a_{ij}|
        \]
        so that $|d_i|=|d_j|$. Since $A$ is irreducible, for any $k,i$ we can find $i_0=k,i_1,\dots, i_l=i$ such that $a_{i_r,i_{r+1}}\neq 0$ and so $|d_k|=|d_{i_1}|=\dots =|d_{i_l}|=|d_i|$. It follows that $D=c\diag(t_1,\dots, t_n)$ for $t_i\in S^1$ and
        \[
            B=DAD^{-1}=cTA\left(\frac{1}{c}T^{*}\right)=TAT^*,
        \]
        where $T=\diag(t_1,\dots, t_n)$.

        If $B^t=DAD^{-1}$, then, since $B$ is Hermitian, $B=B^*=\bar{B^t}=\bar{DAD^{-1}}$. Thus,
        \[
            b_{ij}=\frac{\bar{d_i}}{\bar{d_j}}\bar{a_{ij}}
        \]
        and the same computation as above says that $|d_i|=|d_j|$ for all $i,j$. Hence, $B^t=TAT^{-1}$ for some unitary diagonal matrix $T$.
    \end{proof}
    
\subsection{Simultaneous Permutations}\label{subsec:minors-of-permutations}
    From the definition of irreducibility and cuts, we see that the problem takes a simpler form if we conjugate by a permutation matrix. In the case of a cut $I$ with $|I|=k$, this is apparent by taking a permutation $\sigma$ such that $\sigma([k])=I$. We show that this simplification does not affect the equality on principal minors and can thus be used in the proof of Theorem \ref{thm:same-principal-minors-have-equal-image} without repercussions.
    
    \begin{prop}
        Let $A$ be an $n\times n$ matrix, $P$ a permutation matrix associated to $\sigma\in S_n$, and $S\subset [n]$. Then there exists a permutation matrix $R$ such that
        \[
            (PAP^{-1})[S]=RA[\sigma(S)]R^{-1}.
        \]
        \label{prop:submatrix-of-conjugation}
    \end{prop}
    \vspace{-0.9cm}
    \begin{proof}
        A permutation matrix $P$ associated to a permutation $\sigma\in S_n$ has components $P_{ij}=1$ if $j=\sigma(i)$ and $P_{ij}=0$ otherwise. Write $S=\{s_1<\dots<s_k\}$, so for any $i,j\in [k]$,
        \begin{align*}
            \left(\left(PAP^{-1}\right)[S]\right)_{ij} &=(PAP^{-1})_{s_i,s_j}
            = \sum_{r,s=1}^n P_{s_i,r}A_{r,s}(P^{-1})_{s,s_j}
            =A_{\sigma(s_i),\sigma(s_j)}.
        \end{align*}
        Writing $\sigma(S)=\{t_1<\dots<t_k\}$ and $\sigma(s_i)=t_{\tau(i)}$ for some $\tau\in S_k$, we have
        \begin{align*}
            \left(\left(PAP^{-1}\right)[S]\right)_{i,j} &=A_{\sigma(s_i),\sigma(s_j)}
            =A_{t_{\tau(i)},t_{\tau(j)}}
            =A[\sigma(S)]_{\tau(i),\tau(j)}\\
            &=\sum_{r,s=1}^kR_{i,r}A[\sigma(S)]_{r,s} R_{j,s},
        \end{align*}
        where $R$ is the permutation matrix with $R_{i,\tau(i)}=1$ and $R_{i,j}=0$ otherwise. Thus,
        \[
            (PAP^{-1})[S]=RA[\sigma(S)]R^{-1}.
        \]
    \end{proof}

    \begin{cor}
        If $A$ and $B$ have the same corresponding principal minors, and $P$ is a permutation matrix, then $PAP^{-1}$ and $PBP^{-1}$ have the same principal minors of all orders.
        \label{cor:same-principal-minors-after-permutation}
    \end{cor}
    \begin{proof}            
        If $A$ and $B$ have the same principal minors, then by Proposition \ref{prop:submatrix-of-conjugation},
        \[
            \det\Big((PAP^{-1})[S]\Big)=\det(A[\sigma(S)])=\det(B[\sigma(S)])=\det\Big((PBP^{-1})[S]\Big)
        \]
        for all $S\subset [n]$, as desired.
    \end{proof}


\section{Proof of Theorem \ref{thm:same-principal-minors-have-equal-image}}\label{sec:proof-of-thm-11}
	We prove Theorem \ref{thm:same-principal-minors-have-equal-image} by induction on $n$. 
\subsection{Base Cases $n=1,2,3$}\label{subsec:bases-cases-of-thm-11}
	For $n=1,2,3$, Proposition \ref{prop:Loewy-for-low-n-orthogonal} tells us there exist a diagonal unitary matrix $T$ with either $B=TAT^{-1}$ or $B^t=TAT^{-1}$. In the first case,
    \[
    	|\im B|=|\im (TAT^{-1})|=|TA(\im T^{-1})|=|TA(\CCC^n)|=|T(\im A)|.
    \]
    Since $T$ acts by coordinatewise multiplication of by elements of $S^1$, its action is invariant under the absolute value so
    \[
    	|\im B|=|\im A|.
    \]
    The second case is similar.

	Now assume that Theorem \ref{thm:same-principal-minors-have-equal-image} holds for any two $k\times k$ orthogonal projection matrices with $k<n$ that have the same corresponding principal minors. Let $A$ and $B$ be two $n\times n$ orthogonal projection matrices, $n\ge 4$, that have the same corresponding principal minors.
	We must consider three cases:
	\begin{enumerate}
		\item $A$ is reducible,
		\item $A$ is irreducible and has a cut, and
		\item $A$ is irreducible and has no cuts.
	\end{enumerate}
	We first prove the result for the cases (1) and (3).

\subsection{The Reducible Case}
	Suppose $A$ is reducible. Then there exists a permutation matrix $P$ such that
	\[
		P^{-1}AP=\begin{pmatrix}
			A_1 & 0\\
			0 &  A_2
		\end{pmatrix}
	\]
	where the upper right off-diagonal must also vanish by the Hermitian property, since $P^{-1}=P^t$.
	If $B$ has the same corresponding principal minors of all orders as $A$, then $P^{-1}BP$ has the same corresponding principal minors of all orders as $P^{-1}AP$. By Lemma \ref{lem:components-of-hermitian-with-same-minors}, the same off-diagonal blocks must vanish so that
	\[
		P^{-1}BP=\begin{pmatrix}
			B_1 & 0\\
			0 & B_2
		\end{pmatrix}.
	\]
	Each $B_i$ is also an orthogonal projection defined on a smaller coordinate subspace and its principal minors coincide with those of $A_i$. Hence, by the induction hypothesis, $|\im A_i|=|\im B_i|$. Therefore, since the absolute value is equivariant under permutations,
	\begin{align*}
		|\im A| &=P\left(|\im A_1|\oplus |\im A_2|\right)
		=P\left(|\im B_1| \oplus |\im B_2|\right)
		=|\im B|.
	\end{align*}
	This proves Theorem \ref{thm:same-principal-minors-have-equal-image} for the case where $A$ is reducible.

\subsection{The Irreducible Case with No Cuts}
    Here we are in the setting of Loewy's theorem so, by Proposition \ref{prop:diagonals-in-Loewy-are-unitary}, there exists a diagonal unitary matrix $T$ with either $B=TAT^{-1}$ or $B^t=TAT^{-1}$. By the same computation of Section \ref{subsec:bases-cases-of-thm-11}, we obtain the result.

\subsection{The Irreducible Case with a Cut}
	Suppose now that $A$ is irreducible and has a cut $I$, and write $W_1\coloneqq \im A$. Let $E=W_1\cap \CCC^I$ and $F=W_1\cap \CCC^{I^c}$.
	By Theorem \ref{thm:cut-decomposition}, we can write
	\[
		A[I]=P_E+\alpha uu^*,\quad A[I^c]=P_F+(1-\alpha) vv^*,\quad A[I|I^c]=\sqrt{\alpha(1-\alpha)}uv^*
	\]
	with $\alpha\in (0,1)$, and $u\in \CCC^I\cap E^\perp$, $v\in \CCC^{I^c}\cap F^\perp$ unit vectors, where $P_E,P_F$ are the orthogonal projection matrices onto subspaces $E\subset \CCC^I$, $F\subset \CCC^{I^c}$, respectively. Since simultaneous permutations preserve equality of principal minors (see Corollary \ref{cor:same-principal-minors-after-permutation}) and the absolute value map is equivariant under these, we can assume that $I=[k]$ so
	\[
		A=\begin{pmatrix}
			P_E+\alpha uu^* & \sqrt{\alpha(1-\alpha)}uv^*\\
			\sqrt{\alpha(1-\alpha)}vu^* & P_F+(1-\alpha) vv^*
		\end{pmatrix}.
	\]
	Since $B^2=B$, we have 
	\begin{align*}
		\begin{pmatrix}
			B[I] & B[I|I^c]\\
			B[I^c|I] & B[I^c]
		\end{pmatrix}&=B=B^2=\begin{pmatrix}
			B[I] & B[I|I^c]\\
			B[I^c|I] & B[I^c]
		\end{pmatrix}^2\\
		&=\begin{pmatrix}
			B[I]^2+B[I|I^c]B[I^c|I] & B[I]B[I|I^c]+B[I|I^c]B[I^c]\\
			B[I^c|I]B[I]+B[I^c]B[I^c|I] & B[I^c|I]B[I|I^c]+B[I^c]^2
		\end{pmatrix}
	\end{align*}	
	The equality of the upper left block gives the following.

	\begin{lem}
		The subset $I$ is a cut of $B$.
	\end{lem}
	\begin{proof}
		By Lemma \ref{lem:components-of-hermitian-with-same-minors}, $B$ is irreducible since $A$ is, so we must only show that $I$ is a cut of $B$. Since $A[I]$ and $B[I]$ have the same principal minors, they have the same eigenvalues with multiplicity by Corollary \ref{cor:coeff-of-char-poly}. By Theorem \ref{thm:cut-decomposition}, $\alpha$ is the unique eigenvalue of $A[I]$ in $(0,1)$, counted with multiplicity. Hence, 
		\[
			1=\rank (B[I]-B[I]^2)=\rank B[I|I^c].
		\]
		Since $B$ is Hermitian, the opposite off-diagonal, $B[I^c|I]$, also has rank 1, so $I$ is a cut of $B$.
	\end{proof}
	
	Because of this, Theorem \ref{thm:cut-decomposition} applies so that $B$ has a decomposition similar to that of $A$: write $W_2=\im B$, $M=W_2\cap \CCC^I$, $N=W_2\cap \CCC^{I^c}$, so that
	\[
		B[I]=P_M+\alpha' pp^*,\quad B[I^c]=P_N+(1-\alpha') qq^*,\quad B[I|I^c]=\sqrt{\alpha'(1-\alpha')}pq^*
	\]
	for some $\alpha'\in (0,1)$, and $p\in M^\perp\cap \CCC^I$, $q\in N^\perp \cap \CCC^{I^c}$ unit vectors.
	Now, since $A[I]$ and $B[I]$ have the same eigenvalues, Theorem \ref{thm:cut-decomposition} gives $\alpha=\alpha'$. 
	We keep this notation for the rest of the section and proceed with the proof of Theorem \ref{thm:same-principal-minors-have-equal-image}.

		Introduce new coordinate vectors $e_0,e_0'$ and consider
		\begin{align*}
    		\EE_A &\coloneqq E\oplus \langle \sqrt{\alpha}u+e_0\rangle, \quad \EE_B \coloneqq M\oplus \langle \sqrt{\alpha}p+e_0\rangle\subset \CCC e_0\oplus \CCC^I,\\
    		 \FF_A &\coloneqq F\oplus \langle \sqrt{\beta}v+e_0'\rangle,\quad \FF_B \coloneqq N\oplus \langle \sqrt{\beta}q+e_0'\rangle\subset \CCC e_0'\oplus \CCC^{I^c},
    	\end{align*} 
        where $\beta=1-\alpha$.
    	If we show that $|\EE_A|=|\EE_B|$ and $|\FF_A|=|\FF_B|$, then, by Corollary \ref{cor:absolute-image-of-2-connected-subspace} we would be done. We only prove the first of these equalities as the second is similar. 
    	
    	\begin{lem}
    		Using the notation of Theorem \ref{thm:determinantal-poly-of-full-matrix-with-cut}, we have that
    		\[
    			f_{P_E}=f_{P_M},\qquad f_{P_F}=f_{P_N},\qquad a_{P_E,u}=a_{P_M,p},\qquad a_{P_F,v}=a_{P_N,q}.
    		\]
    	\end{lem}
    	\begin{proof}
    		By Theorem \ref{thm:determinantal-poly-of-full-matrix-with-cut},
	    	\[
	    		f_A=f_{P_E+\alpha uu^*}\cdot f_{P_F+\beta vv^*}-\alpha\beta a_{P_E,u}\cdot a_{P_F,v},
	    	\]
	    	and
	    	\[
	    		f_B =f_{P_M+\alpha pp^*}\cdot f_{P_N+\beta qq^*}-\alpha\beta a_{P_M,p}\cdot a_{P_N,q}.
	    	\]
	    	Since $A$ and $B$ have the same principal minors,
	    	\[
	    		f_A=f_B,\qquad f_{P_E+\alpha uu^*}=f_{P_M+\alpha pp^*},\qquad f_{P_F+\beta vv^*}=f_{P_N+\beta qq^*}. 
	    	\]
	    	Combining these gives
	    	\[
	    		a_{P_E,u}\cdot a_{P_F,v}=a_{P_M,p}\cdot a_{P_N,q}.
	    	\]
	    	Since $a_{P_E,u}$ and $a_{P_M,p}$ are polynomials in $x_i,\ i\in I$, while $a_{P_F,v}$ and $a_{P_N,q}$ are polynomials in $x_j,\ j\in I^c$, we have that
	    	\[
	    		a_{P_E,u}=\lambda a_{P_M,p}, \qquad a_{P_F,v}=\nu a_{P_N,q}, 
	    	\]
	    	for $\lambda,\nu\in \CCC$ nonzero. However, since $a_{P_E,u}=u^*\mathrm{adj}(D_I(x)+P_E)u$, we can expand this to have
	    	\begin{align*}
	    		a_{P_E,u} &= \sum_{i,j\in I} (-1)^{i+j}\bar{u}_i \det\Big((D_I(x)+P_E)[I\setminus\{j\}|I\setminus\{i\}]\Big)u_j\\
	    		&=\sum_{i\in I}\left(|u_i|^2\prod_{j\in I\setminus \{i\}}x_j  \right)+\text{lower order terms}.
	    	\end{align*}
	    	Hence, by comparing coefficients with $a_{P_M,p}$, we have $|u_i|^2=\lambda |p_i|^2$ and summing these together gives $\lambda=1$. Similarly, $\nu=1$.
	    	
	    	With this, we can apply Lemma \ref{prop:determinantal-poly-of-cut} to conclude that
	    	\[
	    		f_{P_E}=f_{P_E+\alpha uu^*}-\alpha a_{P_E,u}=f_{P_M+\alpha pp^*}-\alpha a_{P_M,p}=f_{P_M},
	    	\]
	    	and
	    	\[
	    		f_{P_F}=f_{P_F+\beta vv^*}-\beta a_{P_F,v}=f_{P_N+\beta qq^*}-\beta a_{P_N,q}=f_{P_N}
	    	\]
	    	as desired.
	    \end{proof}

    	The $(|I|+1)\times (|I|+1)$ orthogonal projection matrix onto $\EE_A$ in $\CCC e_0\oplus \CCC^I$ is given by
    	\[
    		P=\def\arraystretch{2}\begin{pmatrix}
    			 \displaystyle\frac{1}{1+\alpha} & \displaystyle\frac{\sqrt{\alpha}}{1+\alpha}u^*\\
    			 \displaystyle\frac{\sqrt{\alpha}}{1+\alpha}u & \displaystyle P_E+\frac{{\alpha}}{1+\alpha} uu^*
    		\end{pmatrix}.
    	\]
    	Similarly, the orthogonal projection matrix onto $\EE_B$ is given by
    	\[
    		Q=\def\arraystretch{2}\begin{pmatrix}
    			 \displaystyle\frac{1}{1+\alpha} & \displaystyle\frac{\sqrt{\alpha}}{1+\alpha}p^*\\
    			 \displaystyle\frac{\sqrt{\alpha}}{1+\alpha}p & \displaystyle P_M+\frac{{\alpha}}{1+\alpha} pp^*
    		\end{pmatrix}.
    	\]
    	We show that $P$ and $Q$ have the same principal minors in order to do induction.

	    \begin{lem}
	    	The orthogonal projection matrices $P$ and $Q$ have the same principal minors of all orders.		
	    \end{lem}
	    \begin{proof}
	    	The principal minor $\det P[S]$ with $0\not\in S$ is a principal minor of the block $P_E+{{\alpha}}/{(1+\alpha)}uu^*$, but by the previous lemma,
	    	\[
	    		f_{P_E+{{\alpha}}/({1+\alpha})uu^*}=f_{P_E}+\frac{{\alpha}}{1+\alpha}a_{P_E,u}=f_{P_M}+\frac{{\alpha}}{1+\alpha}a_{P_M,p}=f_{P_M+{{\alpha}}/({1+\alpha})pp^*}.
	    	\]
	    	So $\det P[S]=\det Q[S]$ for all $S$ not containing $0$.
	    	
	    	Suppose now that $0\in S$ and let $S'=S-\{0\}$. Then, using the Schur determinant formula,
	    	\begin{align*}
	    		\det P[S] &= \det\begin{pmatrix}
	    			\displaystyle\frac{1}{1+\alpha} & \displaystyle \frac{\sqrt{\alpha}}{1+\alpha}u^*[S']\\
	    			\displaystyle \frac{\sqrt{\alpha}}{1+\alpha}u[S'] & \displaystyle \left(P_E+\frac{{\alpha}}{1+\alpha}uu^*\right)[S']
	    		\end{pmatrix}\\
	    		&=\frac{1}{1+\alpha}\det(\left(P_E+\frac{{\alpha}}{1+\alpha}uu^*\right)[S']-\frac{\alpha}{1+\alpha}u[S']u^*[S'])\\
	    		&=\frac{1}{1+\alpha}\det P_E[S'].
	    	\end{align*}
	    	By repeating the computation above for $\det Q[S]$, we find that 
	    	\[
	    		\det Q[S]=\frac{1}{1+\alpha} \det P_M[S']
	    	\]
	    	and since $f_{P_E}=f_{P_M}$, we conclude that $\det P[S]=\det Q[S]$.
	    \end{proof}
	    
	    Since $P$ and $Q$ have equal principal minors of all orders, and they have size $|I|+1\le n-1$ (as $2\le |I|\le n-2$), the induction hypothesis gives us $|\EE_A|=|\EE_B|$.
	    Similarly, $|\FF_A|=|\FF_B|$ and hence $|\im A|=|\im B|$. This concludes the proof of Theorem \ref{thm:same-principal-minors-have-equal-image}.
	    
\section{A Stratification of the Grassmannian}\label{sec:stratification}
	Reducibility of an orthogonal projection matrix $A$ amounts to finding subsets $I\subset [n]$ with 
	\[
		|I|,|I^c|\ge 1 \text{ and } \rank A[I|I^c]<1. 
	\]
	Similarly,
	a cut $I\subset [n]$ of a matrix $A$ is defined by the properties 
	\[
		|I|,|I^c|\ge 2 \text{ and } \rank A[I|I^c]< 2. 
	\]
	This
	points to a general pattern of subsets $I\subset [n]$ with 
	\[
		|I|,|I^c|\ge k \text{ and } \rank A[I|I^c]<k.
	\] 
	In this section, we show that subsets $I$ with the properties above define the notion of $k$-connectivity of a subspace $W=\im A$ and is related to the notion of $k$-connectivity of matroids.
	From this, we obtain that an orthogonal projection matrix $A$ is reducible if and only if $W=\im A$ is exactly 1-connected. Similarly, $A$ is irreducible and has a cut if and only if $W=\im A$ is exactly 2-connected. Continuing in this manner, we obtain a stratification of the Grassmannian given by connectivity levels.

	\subsection{Connectivity of Subspaces}
	Fix an $n\times n$  orthogonal projection matrix $A$ with $\rank A=r$, where $1\le r\le n-1$.
	Recall $A$ is said to be reducible if and only if there is a permutation matrix $P$ such that
	\[
		PAP^{-1} = \begin{pmatrix}
			A_1 & *\\
			0 & A_2
		\end{pmatrix}
	\]
	for some nontrivial square matrices $A_1$ and $A_2$ (i.e., they have size greater than zero). Since $A$ is an orthogonal projection matrix, the upper right corner must also be zero and so, writing $W=\im A$, we have $W=(W\cap \CCC^I)\oplus (W\cap \CCC^{I^c})$ for some nonempty proper subset $I\subset [n]$. We can define a function that measures how far we are from obtaining such a decomposition.
	
	\begin{Def}
		Let $W\subset \CCC^n$ be a subspace. The \textbf{connectivity function} of $W$ is the map
		\[
			\lambda_W(I)=\dim\left(\frac{W}{(W\cap \CCC^I)\oplus (W\cap \CCC^{I^c})}\right),
		\]
		where $I\subset [n]$.
	\end{Def}

	The subspaces $\CCC^I$ and $\CCC^{I^c}$ are orthogonal to each other so the sum in the quotient above is indeed a direct sum. This definition is related to the connectivity function of general matroids studied by Tutte \cite{Tutte1966} and others \cite{costalonga-k-connected, ge2026superminimally3connectedmatroids, oxley2022smallcocircuitsminimallyvertically}. To see this, let $\pi_I:\CCC^n\to \CCC^I$ denote the orthogonal projection so that $\ker \pi_I|_W=W\cap \CCC^{I^c}$. Hence,
	\[
		\dim(W\cap \CCC^{I^c})=\dim \ker\pi_I|_W=\dim W-\dim \pi_I(W).
	\]
	Replacing this expression for $I$ and $I^c$ into the definition of $\lambda_W$ gives us that
	\[
		\lambda_W(I)=\dim \pi_I(W)+\dim \pi_{I^c}(W)-\dim W=r_W(I)+r_W(I^c)-r
	\]
	where $r=\dim W$ and $r_W$ is the rank function of the matroid $M_W$ (see Section 2). Hence, the connectivity function depends purely on the matroid data of $W$ which leads us to the following definition from matroid theory.

    \begin{Def}
    	Let $W\subset \CCC^n$ be a linear subspace. 
    	\begin{enumerate}
		    \item A subset $I\subset [n]$ is called a \textbf{$k$-separation} of $W$ if $|I|,|I^c|\ge k$ and $\lambda_W(I)<k$.
		    \item We say that $W$ is \textbf{$k$-connected} if there are no $j$-separations for $j<k$, and \textbf{exactly $k$-connected} if it is $k$-connected but has a $k$-separation.
		    \item We say that $W$ is \textbf{$\infty$-connected} if it is $k$-connected for all $k\ge 1$.
		    \item The subspace $W$ is said to be \textbf{irreducible} if it is 2-connected and \textbf{reducible} otherwise.
    	\end{enumerate}
    \end{Def}
    
    To see that the connectivity function encodes the rank of $A[I|I^c]$, we begin with a lemma.
    
    \begin{lem}
    	Let $A$ be an $n\times n$ matrix and $I\subset [n]$. Write $\pi_I:\CCC^n\to \CCC^I$ for the orthogonal projection onto $\CCC^I$ and $i_I:\CCC^I\to \CCC^n$ for the canonical inclusion. Then
    	\begin{enumerate}
    		\item $A[I|n]z=(\pi_I\circ A)z$ for all $z\in \CCC^n$,
    		\item $A[n|I]z=(A\circ i_I)(z)$ for all $z\in \CCC^I$, and
    		\item if $A$ is an orthogonal projection matrix, then $A\circ i_I=(\pi_{I}\circ A)^*$.
    	\end{enumerate}
    	Here we have identified the linear maps with their corresponding matrices written with respect to the standard bases.
    \end{lem}
    \begin{proof}
    	For $z\in \CCC^n$, the $i$th component of $A[I|n]z$ is
    	\[
    		(A[I|n]z)_i=\sum_{j=1}^n A_{ij}z_j=\big(\pi_I(Az)\big)_i,
    	\]
    	giving the equality (1).
    	Similarly as above, take $z\in \CCC^I$ so that
    	\[
    		(A[n|I]z)_i=\sum_{j\in I} A_{ij}z_j=\sum_{j\in I} A_{ij}(i_I(z))_j=A(i_I(z))_i,
    	\]
    	proving (2).
    	
    	For (3), since $A$ is an orthogonal projection we have $(A[I|J])^*=A^*[J|I]=A[J|I]$ and hence, using (1) and (2),
    	\[
    		(\pi_I\circ A)^*=(A[I|n])^*=A[n|I]=A\circ i_I.
    	\]
    \end{proof}

    \begin{prop}
    	Let $A$ be the orthogonal projection matrix onto $W$. Then
    	\[
    		\rank A[I|I^c]=\lambda_W(I).
    	\]	
    	\label{prop:principal-minor-rank-is-connectivity-func}
    \end{prop}
    \begin{proof}
    	\vspace*{-2em} Write $\CCC^n=\CCC^I\oplus \CCC^{I^c}$ and let $\pi_I:\CCC^n\to \CCC^I$ be the orthogonal projection and $i_I:\CCC^I\to \CCC^n$ the canonical inclusion. Define
        $M=\pi_I|_W$ and $N=\pi_{I^c}|_W$, so by using the previous lemma, we have
    	\begin{align*}
    		A[I|I^c]z &=\pi_I(A\circ i_{I^c})(z)=(\pi_I|_W)(A\circ i_{I^c})(z)
			=(\pi_I|_W)(\pi_{I^c}|_W)^*(z)=MN^*z
    	\end{align*}
    	for any $z\in \CCC^{I^c}$.
    	We can then compute the rank of $MN^*$:
    	\begin{align*}
    		\rank (MN^*) &= \dim M(\im N^*)\\
    		&=\dim \im M|_{\im N^*}\\
            &=\dim \im N^*-\dim \ker M|_{\im N^*}\\
    		&=\dim (\ker N)^\perp-\dim (\ker M\cap \im N^*)\\
    		&=r-\dim \ker N-\dim (\ker M\cap (\ker N)^\perp).
    	\end{align*}
    	Now, $\ker M$ and $\ker N$ are orthogonal to each other, so $\ker M\subset (\ker N)^\perp$ and hence
    	\begin{align*}
    		\rank (MN^*) &=r-\dim \ker N-\dim \ker M=r-\dim (W\cap \CCC^I)-\dim (W\cap \CCC^{I^c}).
    	\end{align*}
    	The right hand side is $\lambda_W(I)$.
    \end{proof}

    \begin{cor}
        A subset $I$ of $[n]$ is a $k$-separation of $W$ if and only if 
        \[
            k\le |I|\le n-k\qquad  \text{and} \qquad \rank A[I|I^c]<k
        \]
        for $A$ the orthogonal projection matrix associated to $W$.
        In particular, $I$ is a cut of $A$ if and only if $I$ is a $2$-separation of $W$.
    \end{cor}
    
    \begin{cor}
        The orthogonal projection matrix $A$ onto $W$ is irreducible if and only if $W$ is irreducible. In particular, $A$ is irreducible and has a cut if and only if $W$ is exactly 2-connected.
    \end{cor}
    
    \begin{cor}
    	For $A$ an irreducible orthogonal projection matrix onto a subspace $W$, $A$ has no cuts if and only if $W$ is 3-connected.
    	\label{cor:no-cuts-is-3-connected}
    \end{cor}

    	If $n\ge 4$, then the irreducibility assumption can be dropped. That is, for $n\ge 4$, an (irreducible or not) orthogonal projection matrix onto a subspace $W$ has no cuts if and only if $W$ is 3-connected.  This follows from the following.
    	
    \begin{prop}
    	Let $W\subset \CCC^n$ be a subspace and $k$ a positive integer.
    	\begin{enumerate}
    		\item If $2k>n$, then $W$ has no $k$-separations.
    		\item If $2k\le n$ and $W$ has a $j$-separation for some $j\le k$, then it has a $k$-separation.
    	\end{enumerate}
    	\label{prop:j-sep-implies-k-sep}
    \end{prop}
    \begin{proof}
    	(1) If $|I|>n/2$, then $|I^c|< n/2$ so $I$ cannot be a $k$-separation for $2k>n$.
    	
    	(2) Suppose $W$ has a $j$-separation $I$ for $j\le k$. Then $|I|,|I^c|\ge j$ and $\lambda_W(I)< j$. If $|I|,|I^c|\ge k$, then $I$ is a $k$-separation and we are done so suppose without loss of generality that $|I|<k$. Since $n\ge 2k$, we have $|I^c|>n-k\ge k\ge k-|I|>0$ and hence we can take elements $i_1,\dots, i_{k-|I|}\in I^c$. Define $J\coloneqq I\cup \{i_1,\dots, i_{k-|I|}\}$. Then $|J|=|I|+k-|I|=k$ and using the fact that $n\ge 2k$,
    	\begin{align*}
    		|J^c| &= n-|J|=n-|I|-k+|I|\ge k.
    	\end{align*}
    	Furthermore, since $I^c=J^c\cup \{i_1,\dots, i_{k-|I|}\}$, we have that the orthogonal projection $\pi_{J\setminus I}:\CCC^{I^c}\to \CCC^{\{i_1,\dots, i_{k-|I|}\}}$ satisfies
    	\[
    		\ker (\pi_{J\setminus I}|_{W\cap \CCC^{I^c}})=W\cap \CCC^{J^c}
    	\]
    	and hence $\dim(W\cap \CCC^{J^c})\ge \dim(W\cap \CCC^{I^c})-(k-|I|)$. Combining this with the inclusion
    	$W\cap \CCC^I\subset W\cap \CCC^{J}$, we have 
    	\begin{align*}
    		\lambda_W(J) &=\dim W-\dim(W\cap \CCC^J)-\dim(W\cap \CCC^{J^c})\\
    		&\le \dim W-\dim(W\cap \CCC^I)-\dim(W\cap \CCC^{I^c})+(k-|I|)\\
    		&=\lambda_W(I)+k-|I|\\
    		&<j+k-|I|\le k,
    	\end{align*}
    	where the last inequality follows since $|I|\ge j$. We conclude that $J$ is a $k$-separation.
    \end{proof}
    
\vspace{-0.5cm}

	\subsection{The Connectivity Stratification}	
	The above results are a reformulation of the notions of irreducibility and cuts of an orthogonal projection matrix $A$ in terms of the connectivity of its image $W$. These can be summarized as follows:
	\vspace*{0.25cm}
    \begin{empheq}[box=\widefbox]{align*}
            &{A\text{ is reducible}} & \xleftrightarrow{\hspace*{1.5cm}} & &{W\text{ is reducible}}\\
            &{A\text{ is irreducible and has a cut}} & \xleftrightarrow{\hspace*{1.5cm}} & &{W \text{ is exactly $2$-connected}}\\
            &{A\text{ is irreducible and has no cuts}} & \xleftrightarrow{\hspace*{1.5cm}} & &{W \text{ is $3$-connected}}
    \end{empheq}
    \vspace*{0.15cm}
    
    \noindent These are the three different cases we considered in the proof of Theorem \ref{thm:same-principal-minors-have-equal-image}. However, the more general decomposition of all $r$-dimensional subspaces by their connectedness level gives a stratification of the Grassmannian $\Gr(r,n)$ of $r$-dimensional subspaces in $\CCC^n$ in the sense of \cite[\href{https://stacks.math.columbia.edu/tag/09XY}{Tag 09XY}]{StacksProject}, which we recall now.
    
    \begin{Def}
    	A \textbf{stratification} of a topological space $X$ is a partition $\{X_i\}_{i\in I}$ of $X$, indexed by a partially ordered set $I$, such that
    	\begin{enumerate}
    		\item each $X_i$ locally closed, and
    		\item $\bar{X}_i\subset \bigcup_{j\le i}X_j$ for all $i$.
    	\end{enumerate}
    	Each $X_i$ is called a \textbf{stratum} of the stratification.
    \end{Def}

    \begin{lem}
    	The set
    	\[
	    	S_k=\{W\in \Gr(r,n):W\text{ has a }k\text{-separation}\}
	    \]
    	is a finite union of (Zariski) closed subvarieties of $\Gr(r,n)$. In particular, $S_k$ is closed.
    	\label{lem:Sk-is-closed}
    \end{lem}
    \begin{proof}
    	We have that $W\in S_k$ if and only if there exists an $I\subset [n]$ with $|I|,|I^c|\ge k$ such that $\lambda_W(I)< k$. This last condition is equivalent to 
    	\[
    		\dim(W\cap \CCC^I)+\dim (W\cap \CCC^{I^c})> r-k.
    	\]
    	Hence, $W\in S_k$ if and only if $W$ is in the set
    	\[
    		\bigcup_{a+b> r-k} \{W\in \Gr(r,n):\dim(W\cap \CCC^I)\ge a,\ \dim(W\cap \CCC^{I^c})\ge b\}
    	\]
    	for some $I\subset [n]$ with $|I|,|I^c|\ge k$. The union above is over all non-negative integers $a,b$ with $a+b> r-k$. Since each set of the form $\{W\in \Gr(r,n):\dim(W\cap \CCC^K)\ge \ell\}$ is a closed subvariety of $\Gr(r,n)$, as they are Schubert varieties, the set $S_k$ is a union of closed subvarieties, and this union is finite since $a$ and $b$ are bounded above by $r$ (otherwise the sets are just empty).
    \end{proof}
    
    This lemma, combined with Proposition \ref{prop:j-sep-implies-k-sep}, say that the set of exactly $k$-connected subspaces,
    \[
    	C_k\coloneqq S_k\setminus \bigcup_{j=1}^{k-1} S_j=S_k\setminus S_{k-1},\qquad S_0\coloneqq \emptyset,
    \]
    are locally closed.
    
    The remaining subspaces to consider are the $\infty$-connected ones. These are given by the following set
    \[
    	C_\infty=\{W\in \Gr(r,n):W\text{ is $\infty$-connected}\}=\Gr(r,n)\setminus \bigcup_{k=1}^{\lfloor n/2\rfloor} S_k.
    \]
    We note that $C_\infty$ is open and hence dense whenever it is nonempty. We place an order on $\ZZZ_{>0}\cup \{\infty\}$ by declaring $k<\infty$ for all $k\in \ZZZ_{>0}$.
    
    \begin{prop}
    	The family $\{C_k\}_{1\le k\le \lfloor n/2\rfloor}\cup \{C_\infty\}$ is a finite stratification of the Grassmannian $\Gr(r,n)$, where the indexing set $\{1,\dots, \lfloor n/2\rfloor,\infty\}$ has the natural order.
    	\label{prop:stratification-of-grass}
    \end{prop}
    \begin{proof}
    	The sets $C_k$ are disjoint by definition and locally closed as shown above. Since $C_k\subset S_k$ and $S_k$ is closed, it follows that
    	\[
    		\bar{C}_k\subset S_k=\bigcup_{j=1}^k S_j=\bigcup_{j=1}^k (S_j\setminus S_{j-1})=\bigcup_{j=1}^k C_j,
    	\]
    	as desired.
    \end{proof}
    
    We will refer to the stratification given by the sets $C_k$ as the \textbf{connectivity stratification}. It is still open whether this stratification is a ``good stratification'' in the sense of \cite[\href{https://stacks.math.columbia.edu/tag/09XY}{Tag 09XY}]{StacksProject}.

    \begin{rem}
		Our notion of $k$-connectivity is based on the notion of $k$-connectivity of matroids \cite{Oxley2011}. In \cite{GelfandGoreskyMacPhersonSerganova}, Gelfand-Goresky-MacPherson-Serganova (GGMS) define for each subspace $W$ a matroid $M_W$ (see Section \ref{sec:background} for the definition). They then define a stratification of the Grassmannian by saying that $W,W'$ are in the same stratum if and only if $M_W=M_{W'}$. Our definition of exact $k$-connectedness is equivalent to saying that the matroid $M_W$ is exactly $k$-connected. In particular, if two subspaces define the same matroid, they must have the same connectivity level. This means that the stratification above is coarser than the one of GGMS. Furthermore, the open stratum of $\{C_k\}_k\cup \{C_\infty\}$ is the one that contains the subspaces representing the uniform matroid of rank $r$ on $[n]$, as these form a dense open set in $\Gr(r,n)$ with the Zariski topology \cite{GelfandGoreskyMacPhersonSerganova}
		\label{rem:uniform-matroids-are-dense}
	\end{rem}
 	
 	\begin{Def}
 		A subspace $W\subset \CCC^n$ is uniform if $\dim \pi_I(W)=|I|$ for all $I\subset [n]$ with $|I|\le \dim W$. 
 	\end{Def}
 	
 	Equivalently, if we let $r=\dim W$ and $Q$ be an $r\times n$ matrix whose rows form a basis for $W$, then $W$ is uniform if and only if $\det Q[r|S]\neq 0$ for all $S\subset [n]$ with $|S|=r$. As mentioned in Remark \ref{rem:uniform-matroids-are-dense}, these form an open and dense subset of $\Gr(r,n)$.

	We now show that this stratification is not trivial, in the sense that many strata are nonempty.
		
	\begin{Ex}
		Take $W=(W\cap \CCC^I)\oplus (W\cap \CCC^{I^c})$ for some $I$ nonempty proper subset of $[n]$. Then $\lambda_W(I)=0$ so $I$ is a 1-separation. It follows that $W\in C_1$.
	\end{Ex}

	\begin{prop}
		For any $k=1,\dots, \min\{r,n-r\}$, the stratum $C_k$ is nonempty.
		\label{prop:nonempty-strata}
	\end{prop}
	\begin{proof}
		Fix $I\subset [n]$ with $|I|=k$. Let $H$ be an $(r-1)$-dimensional uniform subspace of $\CCC^n$ and let $P$ be the $(r-1)\times n$ matrix whose rows form a basis of $H$. Pick $a\in \CCC^I$ such that the matrix
		\[
			Q=\begin{pmatrix}
				a^t\\
				P
			\end{pmatrix}
		\]
		satisfies $\det Q[r|J]\neq 0$ for any $J\subset [n]$ with $J\cap I\neq \emptyset$ and $|J|=r$. This can be done since for any $J$ with $J\cap I\neq \emptyset$, we can consider the $a_i$ with $i\in I$ as variables, and expanding along the first row gives
		\[
			\det Q[r|J]=\sum_{i\in I\cap J}(\pm a_i)\det P[r-1|J\setminus \{i\}].
		\]
		Each $\det P[r-1|J\setminus \{i\}]$ is nonzero so that the polynomial above is a nonzero linear form. Since there are finitely many subsets $J$ that intersect $I$, the vector $a$ we are looking for is a vector in the complement of the zero set of all such polynomials above. In particular, this is a nonempty set and hence such an $a$ exists.
		
		Let $W$ be the row span of $Q$. We claim that $W$ is exactly $k$-connected. By definition, we have that
		\[
			\dim \pi_J(W)=\begin{cases}
				\min\{r,|J|\}, & \text{ if }J\cap I\neq \emptyset\\
				\min\{r-1,|J|\}, & \text{ if }J\subset I^c.
			\end{cases}
		\]
		Indeed, if $J\subset I^c$, then the above follows by uniformity of $H$. If $J\cap I\neq \emptyset$ and $|J|<r$, then extend it to a set $J'$ of size $r$ so that $r=\dim \pi_{J'}(W)$ meaning that the columns indexed by $J'$ are independent. In particular, those indexed by $J$ are independent and so $\dim \pi_{J}(W)=|J|$. If $J\cap I\neq \emptyset$ and $|J|\ge r$, then take $J'\subset J$ of size $r$ with $J'\cap I\neq \emptyset$ so that $r=\dim \pi_{J'}(W)\le \dim \pi_{J}(W)\le r$. This shows that $\dim \pi_J(W)$ is as described above.
		
		Note that $I$ is a $k$-separation: we have that $|I|=k\le n-r$ so $|I^c|\ge r\ge k$, and
		\[
			\lambda_W(I)=k+r-1-r=k-1<k.
		\] 
		We now show that $W$ has no $\ell$-separation for any $1\le \ell<k$. For this, let $J\subset [n]$ with $|J|,|J^c|\ge \ell$ for some $1\le \ell\le k-1$. 
		
		Suppose first that $J\cap I$ and $J^c\cap I$ are nonempty. Then
		\begin{align*}
			\lambda_W(J)=\min\{r,|J|\}+\min\{r,|J^c|\}-r=\min\{r,n-r,|J|,|J^c|\}.
		\end{align*}
		Since $|J|,|J^c|\ge \ell$, and $\min\{r,n-r\}\ge k>\ell$, we have that $\lambda_W(J)\ge \ell$.
		
		If $J\subset I^c$, then $I\subset J^c$ so that $|J^c|\ge k$.
		We have
		\begin{align*}
			\lambda_W(J) &=\dim \pi_J(W)+\dim \pi_{J^c}(W)-r\\
			&=\min\{r-1,|J|\}+\min\{r,|J^c|\}-r\\
			&=\min\{r-1,n-r,|J|,|J^c|-1\}
		\end{align*}
		By assumption, $|J|\ge \ell$. As mentioned above, $|J^c|\ge k> \ell$ so $|J^c|-1\ge \ell$ and since $\ell <k\le \min\{r,n-r\}$, we have $\min\{r-1,n-r\}\ge \ell$. Hence, $J$ cannot be an $\ell$-separation. The case where $J^c\subset I^c$ is similar.
	\end{proof}

 	\begin{cor}
    	For $k-1\le \min\{r,n-r\}$, the set of all $k$-connected subspaces $W\in \Gr(r,n)$ is open and dense.
    	\label{cor:k-conn-set-is-dense}
    \end{cor}
    \begin{proof}
    	If $k=1$, all subspaces in $\Gr(r,n)$ are $1$-connected and hence the assertion trivially follows. Assume now that $k>1$.
    	
    	By Proposition \ref{prop:j-sep-implies-k-sep}, the set of $k$-connected subspaces is the open set $\Gr(r,n)\setminus S_{k-1}$. If we show this is nonempty we are done. We claim that any uniform subspace $W\in \Gr(r,n)$ is in this set. Write $m=\min\{r,n-r\}$ and suppose $W$ is a uniform subspace which has a $(k-1)$-separation. Since $k-1\le m$ and $2m\le n$, then Proposition \ref{prop:j-sep-implies-k-sep} implies that $W$ has an $m$-separation $I$.
    	Then $|I|,|I^c|\ge m$ and this implies the following inequalities:
    	\begin{align*}
    		m\le |I|\le n-m, \qquad m\le |I^c|\le n-m.
    	\end{align*}
    	Then
    	\[
    		\lambda_W(I)=\begin{cases}
    			r, & \text{ if }m=r\\
    			n-r, & \text{ if }m=n-r.
    		\end{cases}
    	\]
    	In all cases, $\lambda_W(I)\ge m$ so that $I$ cannot be an $m$-separation. Hence, $W$ is in the set $\Gr(r,n)\setminus S_{k-1}$, proving the result.
    \end{proof}

 \subsection{Open Stratum of the Connectivity Stratification}
 	The goal of this section is to find the highest nonempty stratum of the connectivity stratification.
 	Such a stratum $C_k$ must be open, since
 	\[
 		C_k=\Gr(r,n)\setminus \bigcup_{i <k}C_i=\Gr(r,n)\setminus \bigcup_{i< k} S_{i}
 	\]
 	and each $S_{i}$ is closed by Lemma \ref{lem:Sk-is-closed}. It must also be unique since the strata are disjoint and $\Gr(r,n)$ is irreducible.
 	
 	We first recall some facts from matroid theory which we rephrase in the context of vector subspaces. We have included self-contained proof in Appendix \ref{subsec:infty-conn-subs} for completeness.
 	
 	\begin{lem}[\cite{Oxley2011}, Corollary 8.6.3]
 		Let $W\subset \CCC^n$ be a subspace. Then $W$ is $\infty$-connected if and only if it is uniform and one of the following dimensionality conditions holds for $\ell\in \ZZZ_{\ge 0}$:
 		\begin{enumerate}
 			\item $n=2\ell$, $\dim W=\ell$;
 			\item $n=2\ell+1$, $\dim W=\ell$;
 			\item $n=2\ell+1$, $\dim W=\ell+1$.
 		\end{enumerate}
 		\label{lem:infty-conn-matroids-are-uniform}
 	\end{lem}

 	Let $m=\min\{r,n-r\}$. 
 	Because $m\le \lfloor n/2\rfloor$, we either have that $m=\lfloor n/2\rfloor$ or $m<\lfloor n/2 \rfloor$. We claim that these two conditions determine the nonempty open stratum of the connectivity stratification.
 	
 	\begin{thm}
 		Let $m=\min\{r,n-r\}$ as above.
 	\begin{enumerate}
 		\item If $m=\lfloor n/2\rfloor$, then
 		\[
 			\Gr(r,n)=C_\infty\cup \bigcup_{i=1}^m C_i
 		\]
 		with $C_1,\dots, C_m,C_\infty$ nonempty.
 		In this case $C_\infty$ is the nonempty open stratum.
 		
 		\item If $m<\lfloor n/2\rfloor$, then
 		\[
 			\Gr(r,n)=\bigcup_{i=1}^{m+1} C_i
 		\]
 		with $C_1,\dots, C_{m+1}$ nonempty. In this case, $C_{m+1}$ is the nonempty open stratum.
 	\end{enumerate}
 	In either case, the nonempty open stratum consists only of the uniform subspaces.
 	\label{thm:open-stratum}
 	\end{thm}
 	\begin{proof}
 		Suppose first that $m=\lfloor n/2\rfloor$. If $n$ is even, then $m=n/2$ and, by looking at the different values $m$ can take, we conclude that $r=n/2$. By Lemma \ref{lem:infty-conn-matroids-are-uniform}, $C_\infty$ is nonempty as it contains the uniform subspaces. If $n$ is odd, say $n=2\ell+1$ with $\ell\in \ZZZ$ positive, then $r\in \{\ell,\ell+1\}$. In either case, Lemma \ref{lem:infty-conn-matroids-are-uniform} again says that $C_\infty$ is nonempty.  Note that $C_k=\emptyset$ for $k\ge m+1$ since $m+1= \lfloor n/2\rfloor +1> n/2$ so by cardinality reasons, we cannot have $k$-separations for such $k$.
 		
 		Suppose now that $m<\lfloor n/2\rfloor$. In this case $C_\infty$ is empty by Lemma \ref{lem:infty-conn-matroids-are-uniform}. By Corollary \ref{cor:k-conn-set-is-dense} with $k=m+1$, we have that $\Gr(r,n)\setminus S_m$ is open and dense. If we show that $C_k=\emptyset$ for all $k>m+1$, then $\Gr(r,n)\setminus S_m=C_{m+1}$ and hence obtaining the result. Under our assumption on $m$, every $W\in \Gr(r,n)$ admits an $(m+1)$-separation. Indeed, we have
 		\[
 			\lambda_W(I)=\dim \pi_I(W)+\dim \pi_{I^c}(W)-r\le r+r-r=r,
 		\]
 		and
 		\[
 			\lambda_W(I)=\dim \pi_I(W)+\dim \pi_{I^c}(W)-r\le |I|+|I^c|-r=n-r.
 		\]
 		So $\lambda_W(I)\le m$ for all $I\subset [n]$. However, since $m<\lfloor n/2\rfloor$, we can find $I$ such that $|I|,|I^c|\ge m+1$ (indeed, any $I$ with $|I|=m+1$ will do) and then $\lambda_W(I)\le m<m+1$ so $I$ is an $(m+1)$-separation. In particular, no $W\in \Gr(r,n)$ can be $k$-connected for $k>m+1$ and so $C_k=\emptyset$.
 		
 		For the last statement, we know that $C_\infty$ consists only of uniform subspaces by Lemma \ref{lem:infty-conn-matroids-are-uniform}. If $m<\lfloor n/2\rfloor$ and $W$ is not uniform, then there exists some $I\subset [n]$ with $|I|=r$ but $\dim \pi_I(W)\le r-1$. In this case, $|I^c|=n-r$ so that $\dim \pi_{I^c}(W)\le m$ and hence
 		\[
 			\lambda_W(I)=\dim \pi_I(W)+\dim \pi_{I^c}(W)-r\le r-1+m-r=m-1.
 		\]
 		It follows that $I$ is an $m$-separation and so $W\not\in C_{m+1}$. Conversely, the proof of Corollary \ref{cor:k-conn-set-is-dense} shows that uniform subspaces are $(m+1)$-connected. Furthermore, we showed above that any subspace $W\in \Gr(r,n)$ with $m<\lfloor n/2\rfloor$ has an $(m+1)$-separation. We conclude that uniform subspaces are exactly $(m+1)$-connected.
 	\end{proof}
 	
 	We end this paper by showing that our stratification is strictly	 different from the matroid stratification of Gelfand-Goresky-MacPherson-Serganova \cite{GelfandGoreskyMacPhersonSerganova}. Theorem \ref{thm:open-stratum} tells us that the open stratum of both the matroid stratification and the connectivity stratification coincide, but this need not be the case for lower strata. Indeed, the stratum $C_k$ in $\Gr(r,n)$ with $k<m+1$ can have subspaces representing different matroids. 
 	
 	\begin{Ex}
 		Consider $\Gr(r,n)$ with $r\ge 2$. The stratum $C_1$ contains subspaces of the form $W=E\oplus F$ with $E\subset \CCC^I$, $F\subset \CCC^{I^c}$ for some $I\subset [n]$. These could define different matroids: if $I,I^c\neq \emptyset$ and $W=\CCC^I\oplus \CCC^{J}$ with $J\subsetneq I^c$ and $|I|+|J|=r$, then $W$ defines the matroid 
 		\[
 			U_{|I|,I}\oplus U_{|J|,J}\oplus U_{0,[n]-(I\cup J)}, 
 		\]
 		where $U_{r,S}$ is the uniform matroid of rank $r$ on the set $S$. Note that this matroid has loops in the third summand.
 		On the other hand, take $E\subset \CCC^I$ and $F\subset \CCC^{I^c}$ uniform subspaces with $\dim E,\dim F>0$ and $\dim E+\dim F=r$. Then $W'=E\oplus F$ is exactly 1-connected but 
 		\[
 			M_{W'}=U_{\dim E,I}\oplus U_{\dim F,I^c}
 		\]
 		which is not isomorphic to $M_W$ as this matroid has no loops.
 		An example for higher connectivity is in $\Gr(2,5)$. Consider
 		\[
 			A=\begin{pmatrix}
 				1 & 1 & 0 & 1 & 1\\
 				0 & 0 & 1 & 1 & 2
 			\end{pmatrix}, \qquad B=\begin{pmatrix}
 				1 & 1 & 0 & 0 & 1\\
 				0 & 0 & 1 & 1 & 1
 			\end{pmatrix}
 		\]
 		and let $W,W'$ be the row span of $A$ and $B$ respectively. Both $W$ and $W'$ are exactly 2-connected but their matroids $M_{W}$ and $M_{W'}$ are not isomorphic since $M_W$ has only one two-element circuit $I=\{1,2\}$ whereas $M_{W'}$ has two: $J_1=\{1,2\}$ and $J_2=\{3,4\}$.
 	\end{Ex}

\appendix

\section{The Structure of Subspaces in Each Stratum}\label{sec:structure-of-k-conn-subspaces}
	In this appendix we attach self-contained proofs to well-known results in matroid theory from the point of view of linear algebra. We first prove a generalization of Theorem \ref{thm:cut-decomposition} using the cosine-sine decomposition \cite{PaigeWei1994, Stewart1982, doi:10.1137/15M1009573, VanLoan1985} in order to characterize exactly $k$-connected subspaces, and then we give a self-contained proof that $\infty$-connected subspaces must be uniform and of specific dimensions.
	
	\subsection{The Structure of Exactly $k$-Connected Subspaces}
		The structure theorem of Theorem \ref{thm:cut-decomposition} for an exactly 2-connected subspace $W$ relies only on the fact that $W$ has a $2$-separation $I$ with $\lambda_W(I)=1$. In general, the same procedure gives a decomposition for a subspace $W$ that has a $k$-separation with $\lambda_W(I)=k-1$ of the form
			\[
				W=(W\cap \CCC^I)\oplus (W\cap \CCC^{I^c})\oplus \bigoplus_{i=1}^{k-1} \langle \sqrt{\alpha_i}u_i+\sqrt{1-\alpha_i}v_i\rangle
			\]
			where $u_1,\dots, u_{k-1}\in E^\perp\cap \CCC^I$, $v_1,\dots, v_{k-1}\in F^\perp \cap \CCC^{I^c}$ and $\alpha_1,\dots, \alpha_{k-1}\in (0,1)$. This sort of decomposition along a partition $I$ with $\lambda_W(I)=k-1$ is well-known to matroid theorists. We will proceed to derive it again using the cosine-sine decomposition (CSD) and obtain a characterization of exactly $k$-connected using the angles coming from the CSD.
			
			Let $W$ be an $r$-dimensional subspace and let $Q$ be the matrix whose columns form an orthonormal basis for $W$ and let $I,I^c$ be a partition of $[n]$. By grouping together the rows in $I$ and $I^c$ respectively, write
			\[
				Q=\begin{pmatrix}
					Q_1\\
					Q_2
				\end{pmatrix}
			\]
			where $Q_1=Q[I|r]$ and $Q_2=Q[I^c|r]$. Since $Q^*Q=I$, we have that $Q^*_1Q_1+Q_2^*Q_2=I$. Then $Q_1^*Q_1$ and $Q_2^*Q_2$ are commuting Hermitian matrices which can thus be simultaneously diagonalized. Furthermore, the eingevalues of $Q_2^*Q_2$ are $1-\alpha_i$ for each eigenvalue $\alpha_i$ of $Q_1^*Q_1$ and $\alpha_i\in [0,1]$. By writing $\alpha_i=\cos^2(\theta_i)$, the above implies that we can find orthonormal vectors $z_1,\dots, z_r\in \CCC^r$ and $\theta_1,\dots, \theta_r\in [0,\pi/2]$ such that
			\[
				Q_1^*Q_1z_i=\cos^2(\theta_i)z_i,\qquad Q_2^*Q_2z_i=\sin^2(\theta_i)z_i.
			\]
			In particular, $\|Q_1z_i\|=\cos(\theta_i)$ and $\|Q_2z_i\|=\sin(\theta_i)$. Let $u_i\coloneqq Q_1z_i/\|Q_1z_i\|$ and $v_i\coloneqq Q_2z_i/\|Q_2z_i\|$ whenever $\cos(\theta_i)\neq 0$ or $\sin(\theta_i)\neq 0$ respectively. Note that the $u_i$ are mutually orthogonal and same for the $v_i$. Finally, we see that $Qz_1,\dots, Qz_r$ forms an orthonormal basis for $W$ and
			\[
				W\cap \CCC^I=\langle Qz_i:\theta_i=0\rangle, \qquad W\cap \CCC^{I^c}=\langle Qz_i:\theta_i=\pi/2\rangle
			\]
		together with
		\[
			\lambda_W(I)=\#\{i:0<\theta_i<\pi/2\}.
		\]
		This gives us the following.
		
		\begin{prop}
			Let $Q$ be the matrix whose columns form an orthonormal basis for a subspace $W\subset \CCC^n$. For the CSD of $Q$ given by a partition $I,I^c$ of the rows, we have
			\[
				\lambda_W(I)=\#\{i:0<\theta_i<\pi/2 \text{ appearing in the CSD of }Q\}.
			\]		
			In particular, $W$ is $k$-connected if and only if every CSD of $Q$ given by a partition $I,I^c$ has at least $\min\{|I|,|I^c|,k-1\}$ angles in $(0,\pi/2)$. It is exactly $k$-connected if and only if, in addition, there exists $I$ with $|I|,|I^c|\ge k$ and exactly $k-1$ angles in $(0,\pi/2).$
		\end{prop}
		\begin{proof}
			To prove the second part of the statement, suppose there exists a partition $I$, $I^c$ with 
			\[
				\#\{i:0<\theta_i<\pi/2 \text{ appearing in the CSD of }Q\}< \min\{|I|,|I^c|,k-1\}. 
			\]
			Let $j=\min\{|I|,|I^c|,k-1\}$ so that $|I|,|I^c|\ge j$ and $\lambda_W(I)< j$. Then $I$ is a $j$-separation with $j<k$ and hence $W$ is not $k$-connected. Conversely, if $W$ is not $k$-connected, let $I$ be an $\ell$-separation of $W$ with $\ell <k$. Then $k-1,|I|,|I^c|\ge \ell$ so $\min\{|I|,|I^c|,k-1\} \ge \ell$. Since $\lambda_W(I)<\ell\le \min\{|I|,|I^c|,k-1\}$, the CSD of $Q$ along the partition $I,I^c$ has less than $\min\{|I|,|I^c|,k-1\}$ angles in $(0,\pi/2)$.
			
			The last sentence is clear.
		\end{proof}
		
		\begin{cor}
			If $W$ is an exactly $k$-connected space and $I$ is a $k$-separation, then
			\[
				W=(W\cap \CCC^I)\oplus (W\cap \CCC^{I^c})\oplus \bigoplus_{i=1}^{k-1} \langle \cos(\theta_i)u_i+\sin(\theta_i)v_i\rangle,
			\]
			with $u_1,\dots, u_{k-1}\in (W\cap \CCC^I)^\perp\cap \CCC^I$ and $v_1,\dots, v_{k-1}\in (W\cap \CCC^{I^c})^\perp\cap \CCC^{I^c}$ orthonormal vectors, and the angles $\theta_1,\dots, \theta_{k-1}$ lie in $(0,\pi/2)$.
		\end{cor}
	
	The benefit of the decomposition above in the $k=2$ case from Section \ref{sec:Results-from-Lin-Alg} is that it allowed us to easily compute the determinantal polynomial of the orthogonal projection matrix associated to an exactly 2-connected subspace. However, doing the same in the $k>2$ case is not so straightforward for $k\ge 3$ and we will not attempt to do so in this paper.

 \subsection{The Structure of $\infty$-Connected Subspaces}\label{sec:structure-of-max-conn-subspaces}\label{subsec:infty-conn-subs}
 	Every subspace $W$ is $k$-connected for some $k\ge 1$ and we understand the structure of exactly $k$-connected subspaces along a $k$-separation. The last stratum to consider are the $\infty$-connected ones. These are very combinatorial and can only occur in very specific dimensions. The results below are well-known and can be found in \cite{Oxley2011}, but we include them here for completeness.

 	\begin{lem}
 		If $W$ is $\infty$-connected, then $\lfloor n/2\rfloor\le \dim W\le \lceil n/2 \rceil$.
 		\label{lem:dim-of-max-conn-subspace}
 	\end{lem}
 	\begin{proof}
 		Suppose $W$ is $\infty$-connected. Then $\lambda_W(I)\ge \min\{|I|,|I^c|\}$ since otherwise we would have a separation.
 		Let $I\subset [n]$ be such that $|I|\le n/2$. Then
 		\begin{align*}
	 		|I| &=\lambda_W(I)
	 		=\dim W-\dim(W\cap \CCC^{I^c})-\dim(W\cap \CCC^I)\\
	 		&=\dim \pi_I(W)-\dim(W\cap \CCC^I)\le \dim \pi_I(W)\le |I|
	 	\end{align*}
	 	Hence, 
	 	$
	 		|I|=\dim \pi_I(W)\le \dim W
	 	$
	 	for any $|I|\le n/2$, so
	 	\[
	 		\lfloor n/2\rfloor\le \dim W.
	 	\]
	 	For the same $I$, this also gives
	 	\begin{align*}
	 		|I| &=\lambda_W(I)=\dim W-\dim (W\cap \CCC^{I})-\dim(W\cap \CCC^{I^c})\\
	 		&=\dim \pi_I(W)+\dim \pi_{I^c}(W)-\dim W\\
	 		&=|I|+\dim \pi_{I^c}(W)-\dim W.
	 	\end{align*}
	 	It follows that $\dim W=\dim \pi_{I^c}(W)\le |I^c|$ for any $|I^c|\ge n/2$. Thus, $\dim W\le \lceil n/2\rceil$ concluding the proof.
 	\end{proof}

 	The proof of this lemma shows that for a $\infty$-connected subspace $W$, we must have $\dim \pi_I(W)=|I|$ for any $|I|\le \lfloor n/2\rfloor$. 
 	The bounds on the dimension gives us the following cases to consider for $\ell$ a positive integer:
 	 	
 	\vspace{0.25cm}
 	
 	\emph{Case 1: $n=2\ell$.} In this case, $\dim W=n/2=\ell$ and $|I|=\dim \pi_I(W)$ for all $|I|\le \ell$. If $|I|>\ell$, then by dimensionality we must have $\dim \pi_I(W)=\ell$. Hence, $W$ is a uniform subspace of dimension $\ell$ inside $\CCC^{2\ell}$.
 	
 	\vspace{0.25cm}
 	
 	\emph{Case 2: $n=2\ell+1$ and $\dim W= \ell$.} As before, $\dim \pi_I(W)=|I|$ for all $|I|\le \lfloor n/2\rfloor=\ell$ so $W$ is uniform of dimension $\ell$.
 	
 	\vspace{0.25cm}
 	
 	\emph{Case 3: $n=2\ell+1$ and $\dim W=\ell+1$.} We already know that $\dim \pi_I(W)=|I|$ for all $|I|\le \ell$. If $|I|=\ell+1$, then $|I^c|=\ell$ so
 	\[
 		\ell=\lambda_W(I^c)=\dim \pi_I(W)+\ell-(\ell+1)=\dim \pi_I(W)-1.
 	\]
 	So $\dim \pi_I(W)=\ell+1=|I|$ so once again, $W$ is uniform.
 	
 	The converse is also true:
 	
 	\begin{thm}
 		Let $W\subset \CCC^n$ be a subspace. Then $W$ is $\infty$-connected if and only if it is uniform and one of the following dimensionality conditions holds for $\ell\in \ZZZ_{\ge 0}$:
 		\begin{enumerate}
 			\item $n=2\ell$, $\dim W=\ell$;
 			\item $n=2\ell+1$, $\dim W=\ell$;
 			\item $n=2\ell+1$, $\dim W=\ell+1$.
 		\end{enumerate}
 		In all cases, if $I\subset [n]$ is such that $|I|\le |I^c|$, then $\pi_I(W)=\CCC^I$ and $W\cap \CCC^I=0$.
 		\label{thm:infty-conn-matroids-are-uniform}
 	\end{thm}
 	\begin{proof}
 		We have shown above that if $W$ is $\infty$-connected, then it must be of the form described.
 		To prove the converse, suppose $W$ is of the form above. We note that $W$ is uniform of dimension $r$ if and only if $\dim \pi_I(W)=\min\{|I|,r\}$ for all $I\subset [n]$. Then
 		\begin{align*}
 			\lambda_W(I) &= \dim \pi_I(W)+\dim \pi_{I^c}(W)-r\\
 			&=\min\{r,|I|\}+\min\{r,|I^c|\}-r\\
 			&=\min\{r,n-r,|I|,|I^c|\}.
 		\end{align*}
 		By Proposition \ref{prop:j-sep-implies-k-sep}, we only need to show there are no $k$-separations for $k\le n/2$. For $(r,n)=(\ell,2\ell),(\ell,2\ell+1),(\ell+1,2\ell+1)$ this means $k\le \ell$. Note that, for all such $(r,n)$, we have $\min\{r,n-r\}\ge \ell\ge k$. Hence, for any $I$ with $|I|,|I^c|\ge k$, we have
 		\[
 			\lambda_W(I)=\min\{r,n-r,|I|,|I^c|\}\ge k
 		\]
 		so that $I$ cannot be a $k$-separation. Therefore, $W$ is $\infty$-connected.
 	\end{proof}
 	
\bibliographystyle{plain}
\bibliography{references}

\end{document}